\documentclass[a4paper,12pt]{amsart}

\usepackage{latexsym}
\usepackage{amssymb}
\usepackage{graphicx}
\usepackage{wrapfig}
\usepackage{amsthm}
\usepackage{float}
\usepackage{epsfig}
\RequirePackage{color}

\usepackage{amscd,enumerate}
\usepackage{amsfonts}

\input xy
\xyoption{all}

\usepackage{multicol}
\usepackage{hyperref}
\hypersetup{ colorlinks,
linkcolor=blue,
filecolor=green,
urlcolor=blue,
citecolor=blue }

\newtheorem{lemma}{Lemma}[section]
\newtheorem{corollary}[lemma]{Corollary}
\newtheorem{theorem}[lemma]{Theorem}
\newtheorem{proposition}[lemma]{Proposition}
\newtheorem{remark}[lemma]{Remark}
\newtheorem{definition}[lemma]{Definition}
\newtheorem{definitions}[lemma]{Definitions}
\newtheorem{example}[lemma]{Example}
\newtheorem{examples}[lemma]{Examples}

\newcommand{\Ker}{{\rm{Ker}}}
\newcommand{\dsum}{\sum}

\newcommand{\N}{{\mathbb{N}}}
\newcommand{\uloopr}[1]{\ar@'{@+{[0,0]+(-4,5)}@+{[0,0]+(0,10)}@+{[0,0] +(4,5)}}^{#1}}

\definecolor{turquoise2}{rgb}{0,0.898039,0.933333}
\definecolor{magenta}{rgb}{1,0,1}
\definecolor{olivedrab}{rgb}{0.419608,0.556863,0.137255}
\definecolor{purple2}{rgb}{0.568627,0.172549,0.933333}

\begin{document}


\title[Evolution algebras of arbitrary dimension and their decompositions]{Evolution algebras of arbitrary dimension and their decompositions}

\author[Y. Cabrera]{Yolanda Cabrera Casado}
\address{Y. Cabrera Casado: Departamento de \'Algebra Geometr\'{\i}a y Topolog\'{\i}a, Fa\-cultad de Ciencias, Universidad de M\'alaga, Campus de Teatinos s/n. 29071 M\'alaga.   Spain.}
\email{yolandacc@uma.es}

\author[M. Siles ]{Mercedes Siles Molina}
\address{M. Siles Molina: Departamento de \'Algebra Geometr\'{\i}a y Topolog\'{\i}a, Fa\-cultad de Ciencias, Universidad de M\'alaga, Campus de Teatinos s/n. 29071 M\'alaga.   Spain.}
\email{msilesm@uma.es}

\author[M. V. Velasco]{M. Victoria Velasco}
\address{M. V. Velasco:  Departamento de An\'{a}lisis Matem\'{a}tico, Universidad de Granada, 18071 Granada, Spain.}
\email{vvelasco@ugr.es}

\begin{abstract}
We study evolution algebras of arbitrary dimension. We analyze in deep the notions of evolution subalgebras,  ideals and non-degeneracy and describe the ideals generated by one element and characterize the simple evolution algebras. We also prove the existence and unicity of a direct sum decomposition into irreducible components for every non-degenerate evolution algebra. When the algebra is degenerate, the uniqueness cannot be assured.

The graph associated to an evolution algebra (relative to a natural basis) will play a fundamental role to describe the structure of the algebra. Concretely, a non-degenerate evolution algebra is irreducible if and only if the graph is connected. Moreover,  when the evolution algebra is finite-dimensional, we give a process (called the fragmentation process) to decompose  the algebra into irreducible components.
\end{abstract}
\maketitle

\medskip

\section{Introduction }

The mathematical study of the genetic inheritance began in 1856 with the
works of Gregor Mendel, who was a pioneer in using mathematical notation to
express his genetics laws. After relevant contributions of authors as
Jennings (1917), Serebrovskij (1934) and Glivenko (1936), to give an
algebraic interpretation of the sign $\times $ of sexual reproduction, a
precise mathematical formulation of Mendel's laws in terms of non-associative
algebras was finally provided in the well known papers \cite{Ethe1,Ethe2}.
Since then, many works pointed out that non-associative algebras are an
appropriate mathematical framework for studying Mendelian genetics
\cite{Bertrand,Reed,Tian,Word-Bus}. Thus, the term \emph{genetic algebra} was
coined to denote those algebras (most of them non-associative) used to
model inheritance in genetics.
\medskip

Recently a new type of genetic algebras,  denominated \emph{evolution
algebras}, has emerged to enlighten the study of non-Mendelian genetics,
which is the basic language of the molecular Biology. In particular, evolution
algebras can be applied to the inheritance of organelle genes, for instance,
to predict all possible mechanisms to establish the homoplasmy of cell
populations. The theory of evolution algebras was introduced by Tian in \cite{Tian},
a pioniering monograph where many connections of evolution algebras with
other mathematical fields (such as graph theory, stochastic processes, group
theory, dynamic systems, mathematical physics, etc) are established.
In this book it is shown
the close connection between evolution algebras, non-Mendelian genetics and
Markov chains, pointing out some further research topics.
Algebraically, evolution algebras are non-associative  algebras (which
are not even power-associative), and dynamically they represent discrete
dynamical systems. An \emph{evolution algebra} is nothing but a
finite-dimensional algebra $A$ provided with a basis
$B=\{e_{i} \ \vert \ i\in \Lambda \},$ such that $e_{i}e_{j}=0,$ whenever $
i\neq j$ (such a basis is said to be natural). If $e_{k}^{2}=\dsum_{i\in \Lambda }\omega
_{ik}e_{i},$ then the coefficients $\omega _{ij}$ define the named
structure matrix $M_{B}$ of $A$ relative to $B$  that codifies the
dynamic structure of $A.$

In \cite{Tian}, evolution algebras are associated to free populations to give
the explicit solutions of a nonlinear evolutionary equation in the absence
of selection, as well as general theorems on convergence to equilibrium in
the presence of selection. In the last years, many different aspects of the
theory of evolution algebras have seen  considered. For instance, in \cite{evo2}
evolution algebras are associated to function spaces defined by Gibbs
measures on a graph, providing a natural introduction of thermodynamics in
the study of several systems in biology, physics and mathematics. On the
other hand, chains of evolution algebras (i.e. dynamical systems the state of
which at each given time is an evolution algebra) are studied in
\cite{evo1,evo10,evo11,evo15}. Also the derivations of some evolution algebras have
been analyzed in \cite{Tian,evo3,evo4}. In \cite{evo4}, the evolution
algebras have been used to describe the inheritance of a bisexual population
and, in this setting, the existence of non-trivial homomorphisms onto the
sex differentiation algebra have been studied in \cite{evo14}. Algebraic
notions as  nilpotency and  solvability may be interpreted
biologically as the fact that some of the original gametes (or generators)
become extinct after a certain number of generations, and these algebraic properties have
been studied in \cite{evo5,evo6,evo7,evo8,evo9,evo12,evo13}.

\medskip

Once we have given a general overview of evolution algebras, we start to explain the results in this work.
\medskip

Since evolution algebras appear after mendelian algebras, it is natural to ask if these  are evolution algebras. The answer is no, as we show in Example \ref{MendEvol}. The fact that evolution algebras are not mendelian algebras is known.

In this paper we deal with evolution algebras of arbitrary dimension. In Section \ref{basic} we study the notions of subalgebra and ideal and explore when they have natural bases and when their natural bases can be extended  to the whole algebra (we call this the extension property) and provide examples in different situations. We also show that the class of evolution algebras is closed under quotients and under homomorphic images, but not under subalgebras or ideals and give an example of a homomorphism of evolution algebras whose kernel is not an evolution ideal.
The aim of the second part of this section is the study of non-degeneracy. We show that this notion, which is given in terms of a fixed natural basis of the algebra,  does not depend on the election of the basis (Corollary \ref{MVM}). A radical is introduced (the intersection of all the absorption ideals) which is zero if and only if the algebra is non-degenerate (Proposition \ref{Relacion}).
The classical notions of semiprimeness and nondegeneracy are also studied and compared to that of non-degeneracy (see Proposition \ref{nilpotenciaIdealesEvolucion} and the paragraph before).
In the last part of this section we associate a graph to any evolution algebra. This has been done yet in the literature, although for finite dimensional evolution algebras. The use of the graph will allow to see in a more visual way properties of the evolution algebra. For example, we can detect the annihilator of an evolution algebra by looking at its graph (concretely determining its sinks) and we can say when a non-degenerate evolution algebra is irreducible (as we explain below).

In Section \ref{UnElemento} we use the graph representation and the notion of descendent to describe the ideal generated by any element in an evolution algebra (Proposition \ref{cuadrado}) and show that its dimension as a vector space is at most countable (Corollary \ref{dimensGeneral}). This  implies that any simple algebra has dimension at most countable.

Section \ref{Simple} is devoted to the study and characterization of simple evolution algebras (Proposition \ref{simple} and Theorem \ref{CharacSimple}).  We also provide examples to show that the conditions in the characterizations cannot be dropped. We finish the section with the characterization of finite dimensional simple evolution algebras (Corollary \ref{CharSimpleFin}).

The direct sum  of a certain number of evolution  algebras is an evolution algebra in a canonical way.
In Section \ref{DirectSum} we deal with the question of when a non-zero evolution
algebra $A$ is the direct sum of  non-zero evolution subalgebras. In particular, an evolution algebra with an associated graph (relative to a certain natural basis) which is not connected is reducible (see Proposition \ref{CompCon}). Next, Theorem \ref{caracteriz}  characterizes the decomposition of a non-degenerate evolution algebra into subalgebras (equivalently ideals) in terms of the elements of any natural basis. We also are interested in determining when every component in a direct sum is irreducible. In Corollary \ref{conexo} we prove that a non-degenerate evolution algebra is irreducible if and only if the associated graph (relative to any natural basis) is connected.

 The decomposition $A= \oplus_{\gamma\in \Gamma}I_\gamma$ of an evolution algebra into irreducible ideals (called an optimal decomposition) exists and is unique whenever the algebra is non-degenerate (Theorem \ref{uniopti}). To assure the uniqueness, this hypothesis cannot be eliminated  (Example \ref{new}).

To get a direct-sum decomposition of a finite dimensional evolution algebra  we identify in the associated graph (relative to a natural basis) the principal cycles and the chain-start indices through the fragmentation process (Proposition \ref{uni-fragme}). This provides an optimal direct-sum decomposition,
which is unique, when the algebra is non-degenerate, as shown in Theorem \ref{final}.

In the last section of the paper we provide a program with Mathematica to
obtain the optimal fragmentation of a natural basis. From this we get a direct-sum decomposition of a reducible evolution algebra starting from its structure matrix.

We have tried to translate the mathematical concepts into biological meaning.

\section{Basic facts about evolution algebras}\label{basic}

Before  introducing evolution algebras we establish in a precise way what we
mean by an algebra in this paper. An \emph{algebra} is a vector space $A$ over a
field $\mathbb{K}$, provided with a bilinear map $A\times A\rightarrow A$ given by
$(a,b)\mapsto ab,$ called  the \emph{multiplication} or the \emph{product} of $A.$
An algebra $A$  such that $ab=ba$ for every $ a,b\in A$  will be called \emph{commutative}.
If $(ab)c=a(bc)$ for every $a,b,c\in A,$ then we say that $A$ is \emph{associative}.
We recall that an algebra $A$ is \emph{flexible} if $a(ba)=(ab)a$ for
every $a,b\in A.$  \emph{Power associative} algebras are those  such that
the subalgebra generated by an element is associative. Particular cases of flexible
algebras are the commutative and also the associative ones.

\medskip

The theory of evolution algebras  appears in the study of non-Mendelian inheritance
(which is essential for molecular genetics). This is the case, for example, of the
bacterial species Escherichia coli because their reproduction is asexual.
In particular, evolution algebras model population genetics (which is the study
of the frequency and interaction of alleles and genes in populations) of
organelles (specialized subunits within a cell that have a specific
function) as well as organisms such as the Phytophthora infestans (an
oomycete that causes the serious potato disease known as late blight or
potato blight, and which also infects tomatoes and some other members of the
Solanaceae).

Let us consider a population of organelles in a cell or a cell clone, and
suppose that $e_{1},\dots, e_{n}$ are $n$ different genotypes in the organelle
population. By the non-mendelian inheritance the
crossing of genotypes is impossible since it is uniparental inheritance.
Thus $e_{i}e_{j}=0$ for every $i\neq j.$ On the other hand, intramolecular and
intermolecular recombination within a lineage provides evidence that one
organelle genotype could produce other different genotypes. Consequently:
$$
\begin{matrix}
e_{i}e_{i}=\sum\limits_{k=1}^{n}\omega _{ki}e_{k},
\end{matrix}
$$

\noindent
where $\omega _{ki}$ is a positive number that can be interpreted as the
rate of the genotype $e_{k}$ produced by the genotype $e_{i}$ (see \cite[pp.
9, 10]{Tian}). Therefore, as pointed out in \cite{Tian}, the next definition models all
the non-Mendelian inheritance phenomena.
\medskip

\begin{definitions}
\rm
An \emph{evolution algebra} over a field $\mathbb K$ is a $\mathbb K$-algebra $A$ provided with
a basis $B=\{e_{i} \ \vert \ i\in \Lambda \}$ such that $e_{i}e_{j}=0$
whenever $i\neq j$.
 Such a basis $B$ is called a \emph{natural basis}.
Fixed a natural basis $B$ in $A,$ the scalars $\omega _{ki}\in \mathbb K$ such that
 $e_{i}^{2}:=e_ie_i=\dsum_{k\in \Lambda} \omega _{ki}e_{k}$ will be called the \emph{structure constants} of $A$\emph{\ relative to} $B$, and the matrix $M_B:= \left(w_{ki}\right)$ is said to be the \emph{structure matrix of} $A$ \emph{relative to} $B$. We will write $M_B(A)$ to emphasize the evolution algebra we refer to.
Observe that  $\vert \{k\in\Lambda \vert \ \omega_{ki}\neq 0\}\vert < \infty$ for every $i$, therefore $M_B$ is a matrix in ${\rm CFM}_\Lambda(\mathbb K)$, where ${\rm CFM}_\Lambda(\mathbb K)$ is the vector space of those
 matrices (infinite or not) over $\mathbb K$ of size $\Lambda \times \Lambda$ for which every column has at most a finite number of non-zero entries.
\end{definitions}

According to \cite{Tian}, the product $e_{i}e_{i}$, where $e_i$ is in a finite dimensional natural basis, mimics the self-reproduction of alleles
in non-Mendelian genetics.

\medskip

Note that  an $n$-finite dimensional  algebra $A$ is an evolution algebra if and only if there is a
basis $B=\{e_{1},...,e_{n}\}$  relative to which the multiplication
table is diagonal:
\begin{equation*}
\begin{tabular}{c|ccc}
& $e_{1}$ & $...$ & $e_{n}$ \\ \hline
$e_{1}$ & $\dsum\limits_{k=1}^{n}\omega _{k1}e_{k}$ & $0$ & $
\overset{}{0}$ \\
$\vdots $ & $0$ & $\dsum\limits_{k=1}^{n}\omega _{ki}e_{k}$ & $0$
\\
$e_{n}$ & $0$ & $0$ & $\dsum\limits_{k=1}^{n}\omega _{kn}e_{k}$
\end{tabular}
\end{equation*}

In this case, the structure matrix of the evolution algebra $A$
relative to the natural basis $B$ is the following one:
\begin{equation*}
M_B=\left(
\begin{array}{ccc}
\omega _{11} & \ldots & \omega _{1n} \\
\vdots & \ddots & \vdots \\
\omega _{n1} & \cdots & \omega _{nn}
\end{array}
\right) \in M_n(\mathbb K).
\end{equation*}

Every evolution algebra is uniquely determined by its structure matrix: if $A$ is an evolution algebra and $B$ a natural basis of $A$, there is a matrix, $M_B$, associated to $B$ which represents the product of the elements in this basis. Conversely, fixed a basis
$B=\{e_i\ \vert \ i\in \Lambda\}$ of a $\mathbb K$-vector space $A,$
each matrix in $ {\rm CFM}_\Lambda(\mathbb K)$ defines a product in $A$ under which $A$ is an
evolution algebra and $B$ is a natural basis.
\medskip

Now we compute the formula of the product of any two elements in an evolution algebra.
Let $A$ be an evolution algebra and $B=\{e_{i}\ \vert\  i\in \Lambda \}$ a natural
basis. Consider elements $a=\sum_{i\in \Lambda}\alpha _{i} e_{i}$ and
$b= \sum_{i\in \Lambda }$ $\beta _{i}$$e_{i}$ in $A$, with $\alpha_i, \beta_i \in \mathbb K$. Then:

\begin{equation*}
ab=\sum_{i\in \Lambda }\alpha _{i}\beta _{i}e_{i}^{2}\newline
=\sum_{i\in \Lambda }\alpha _{i}\beta _{i}\left( \sum_{j\in \Lambda
}\omega_{ji}e_{j}\right) =\sum_{i,j\in \Lambda }\alpha _{i}\beta _{i}\omega
_{ji}e_{j}.
\end{equation*}

We use this computation to produce examples of evolution algebras which are not power-associative. Indeed, from the equation above we have

\begin{equation*}
e_{i}^{2}e_{i}^{2}=\dsum_{k\in \Lambda}\omega _{ki}e_{k}\dsum_{j\in \Lambda}\omega
_{ji}e_{j}=\dsum_{k\in \Lambda}\omega _{ki}^{2}e_{k}^{2}
\end{equation*}
and

\begin{equation*}
(e_{i}^{2}e_{i})e_{i}=\left(\left(\dsum_{k\in \Lambda}\omega _{ki}e_{k}\right)e_i\right)e_{i}=\omega
_{ii}e_{i}^{2}e_{i}=\omega _{ii}\left(\dsum_{k\in \Lambda}\omega _{ki}e_{k}\right)e_{i}=\omega
_{ii}^{2}e_{i}^{2}.
\end{equation*}

Thus every matrix $(\omega_{ki})\in  {\rm CFM}_\Lambda(\mathbb K)$ such that $\omega_{ki}^2 \neq 0$, with $k\neq i$, gives an example of an evolution algebra which is not power-associative. In fact, the only evolution algebras which are power-associative are those such that $w_{ii}^2=w_{ii}$ for every $i$.
Consequently, evolution algebras are not, in general, Jordan, alternative or associative algebras. Evolution algebras are not Lie algebras either. However, by definition, every evolution algebra is commutative and, hence, flexible.
\medskip

We said at the beginning of this section that evolution algebras are the language of non-Mendelian genetics. The next example shows that the class of algebras modeling Mendel's laws are not included in the class of evolution algebras. More precisely we will see that the zygotic algebra for simple Mendelian inheritance for one gene with two alleles, A and a, is not an evolution algebra. For this algebra, according to Mendel laws, zygotes have three possible genotypes, namely: $AA$, $Aa$ and $aa$. The rules of simple Mendelian inheritance are expressed in the  multiplication table included in the example that follows (see \cite{Reed} for details and  similar examples of algebras following Mendel's laws).

\begin{example}\label{MendEvol}
\rm
Consider the vector space generated by the basis  $B=\{AA,Aa,aa\}$
provided with the multiplication table given by:

\begin{equation*}
\begin{tabular}{c|ccc}
& $AA$ & $Aa$ & $aa$ \\ \hline
$AA$ & $AA$ & $\overset{}{\frac{1}{2}AA+\frac{1}{2}Aa}$ & $Aa$ \\
$Aa$ & $\frac{1}{2}AA+\frac{1}{2}Aa$ & $\frac{1}{4}AA+\frac{1}{4}aa+\frac{1}{
2}Aa$ & $\overset{}{\frac{1}{2}aa+\frac{1}{2}Aa}$ \\
$aa$ & $Aa$ & $\overset{}{\frac{1}{2}aa+\frac{1}{2}Aa}$ & $aa$
\end{tabular}
\end{equation*}

We claim that this is not an evolution algebra.
\end{example}
\begin{proof}[Proof of the claim.]
We will see that this algebra does not have a natural basis. Suppose on the contrary that there exists  a natural basis $B'=\{e_{i} \ \vert \ i\in \{1,2,3\} \}$. For each $i \in \{1,2,3\}$ we may write $e_{i}=\alpha_{1i} AA + \alpha_{2i} Aa + \alpha_{3i} aa$.

Since $e_ie_j=0$ for each $i,j \in \{1,2,3\}$, with $i\neq j$, we have that:

$$
\begin{matrix}
\nonumber  4 \alpha_{1j} \alpha_{1i} + 2\alpha_{1j}\alpha_{2i} +  2\alpha_{2j}\alpha_{1i} +  \alpha_{2j}\alpha_{2i} & =  0 \\
 \alpha_{1i} \alpha_{2j} + 2\alpha_{1i}\alpha_{3j}+  \alpha_{2i}\alpha_{1j} +  \alpha_{2i}\alpha_{2j}+ \alpha_{2i}\alpha_{3j} + 2 \alpha_{3i}\alpha_{1j} + \alpha_{3i}\alpha_{2j} & = 0 \\
 4 \alpha_{3i} \alpha_{3j} + 2\alpha_{3i}\alpha_{2j} +  2\alpha_{2i}\alpha_{3j} +  \alpha_{2i}\alpha_{2j} & =  0 \\
\end{matrix}
$$
for every $i,j \in \{1,2,3\}.$
\medskip

We can express these three equations as:
$$
\begin{matrix}
 \nonumber (2\alpha_{1i}+\alpha_{2i})(2\alpha_{1j}+\alpha_{2j}) & = 0 \\
\nonumber (2\alpha_{3i}+\alpha_{2i})(2\alpha_{3j}+\alpha_{2j}) & =  0 \\
\alpha_{1i}\alpha_{2j}+2\alpha_{1i}\alpha_{3j}+  \alpha_{2i}\alpha_{1j} + \alpha_{2i}\alpha_{2j}+ \alpha_{2i}\alpha_{3j} + 2 \alpha_{3i}\alpha_{1j} + \alpha_{3i}\alpha_{2j} & =  0
\end{matrix}
$$

Since these identities hold for every $i,j \in \{1,2,3\}$,  the only option is that there exist $m$, $n$, $s$  $\in \{1,2,3\}$, with $m \neq n$ and $m \neq s$, such that:

\begin{eqnarray}\label{equation1}
\left\{\quad \begin{matrix}
2\alpha_{1m}+\alpha_{2m} & =  0 \\
2\alpha_{1n}+\alpha_{2n} & =  0 \\
2\alpha_{3m}+\alpha_{2m} & =  0 \\
2\alpha_{3s}+\alpha_{2s} & =  0
\end{matrix}
\right.
\end{eqnarray}

Now, we distinguish two cases:

\begin{description}
\item[Case 1] Suppose that $n \neq s$.
From (\ref{equation1}) we obtain that:
$$
\begin{matrix}
\alpha_{1m} =&\alpha_{3m}\\
\alpha_{2m} =&-2\alpha_{1m}\\
 \alpha_{2n} =&-2\alpha_{1n}\\
 \alpha_{2s} =&-2\alpha_{3s}
\end{matrix}
$$

It follows:
 $$
 \begin{matrix}
 v_m  =&\alpha_{1m}e_1-2\alpha_{1m}e_2+\alpha_{3m}e_3 \\
 v_n =& \alpha_{1n}e_1-2\alpha_{1n}e_2+\alpha_{3n}e_3 \\
v_s = &\alpha_{1s}e_1-2\alpha_{3s}e_2+\alpha_{3s}e_3 \\
 \end{matrix}
$$

On the other hand, if we take $i=n$ and $j=s$ in (1), then
$$ \alpha_{1n}\alpha_{2s}+2\alpha_{1n}\alpha_{3s}+  \alpha_{2n}\alpha_{1s} + \alpha_{2n}\alpha_{2s}+ \alpha_{2n}\alpha_{3s} + 2 \alpha_{3n}\alpha_{1s} + \alpha_{3n}\alpha_{2s} =  0;$$
\noindent
this means that $\alpha_{1n}=\alpha_{3n}$ or $\alpha_{3s}=\alpha_{1s}$.
  In any case, this is impossible due to the fact that $v_m$, $v_n$ and $v_s$ are linearly independent.

\item[Case 2] Suppose that $n=s$, then:
 $$
 \begin{matrix}
  \alpha_{1m} =& \alpha_{3m} \\
  \alpha_{1n} =& \alpha_{3n} \\
  \alpha_{2m}=& -2\alpha_{1m}\\
  \alpha_{2n}=& -2\alpha_{1n}
  \end{matrix}
  $$
  \noindent
 which is impossible because $v_{m}$ and $v_n$ are linearly independent.
  \end{description}
\end{proof}


\subsection{Subalgebras and ideals of an evolution algebra}


In this section we study the notions of evolution subalgebra and evolution ideal.
We will see that the class of evolution algebras is not closed neither under subalgebras (Example \ref{Ejtian})
nor under ideals (Example \ref{evolu-ideal}). This last example also shows that the kernel of a homomorphism
between evolution algebras is not necessarily an evolution ideal (contradicting \cite[Theorem 2, p. 25]{Tian}).

\begin{example}
\emph{\cite[Example 1.2]{Tian-vo}.}\label{Ejtian}
\rm
Let $A$ be the evolution
algebra with natural basis $B=\{e_{1},e_{2},e_{3}\}$ and
multiplication table given by $e_{1}^{2}=e_{1}+e_{2}=-e_{2}^{2}$
and $e_{3}^{2}=-e_{2}+e_{3}.$ Define $u_{1}:=e_{1}+e_{2}$
and $u_{2}:=e_{1}+e_{3}.$ Then the subalgebra generated
by $u_{1}$ and $u_{2}$ is not an evolution algebra as follows. Suppose on the
contrary that there exist $\alpha ,\beta ,\gamma ,\delta \in
\mathbb{K}$ such that $v_{1}=\alpha u_{1}+\beta u_{2} $ and
$v_{2}=\gamma u_{1}+\delta u_{2}$ determine a natural basis of the considered subalgebra.
Since
$
v_{1}v_{2}=(\alpha u_{1}+\beta u_{2})(\gamma u_{1}+\delta u_{2})=(\alpha \delta +\beta
\gamma )u_{1}+\beta \delta u_{2}
$, the identity $v_{1}v_{2}=0$ and the linear independency of $u_1$ and $u_2$ imply that $v_{1}$ and $v_{2}$ are linearly
dependent, a contradiction.

\end{example}

\medskip

Because a subalgebra of an evolution algebra does not need to be an
evolution algebra it is natural to introduce the notion of evolution subalgebra.
In \cite[Definition 4, p. 23]{Tian} (and also in \cite{Tian-vo}),
an evolution subalgebra of an evolution algebra $A$ is defined as a subspace $A'$, closed under the product of $A$ and
endowed with a natural basis $\{e_{i} \ \vert \ i\in \Lambda'\}$ which can be extended to a natural basis $B=\{e_{i} \ \vert \ i\in \Lambda \}$ of $A$ with $\Lambda'\subseteq \Lambda$. Nevertheless, we prefer to introduce the following new definition of evolution subalgebra.

\begin{definitions}\label{subal}
\rm
An \emph{evolution subalgebra} of an evolution algebra $A$
is a subalgebra $A' \subseteq A$ such that $A'$
is an evolution algebra, i.e. $A'$ has a natural basis.

We say that $A'$ has the \emph{extension property} if there exists a natural basis $B'$ of $A'$ which can be extended to a natural basis of $A$.
\end{definitions}

\begin{remark}
\rm

Let $A$ be an evolution algebra with basis $\{e_{i} \ \vert \ i\in \Lambda'\}$. As it was said before every element $e_i$ can be interpreted as a genotype. A linear combination $\sum\limits_{i\in \Lambda}\alpha_ie_i$ can be seen as a single individual such that the frequency of having genotype $e_i$ is $\alpha_i$.

A subalgebra $A'$ of $A$ is a population consisting of single individuals, each of which has a certain frequency of having genotype $e_i$ and such that its reproduction (i.e. its product) remains in $A'$.

An evolution subalgebra $A'$ will have the extension property if there exist genotype sets $B'$ and $B''$ of $A$ such that $B'$ is a natural basis of $A'$ and $B'\cup B''$ is a natural basis of $A$.
\end{remark}

Note that an evolution subalgebra in the meaning of \cite{Tian} is an evolution subalgebra in the sense of Definitions \ref{subal} having the extension property.
Thus, this last definition of evolution subalgebra is natural and less restrictive as  Example
\ref{hui} below proves (where we give an ideal $I$ which is an evolution algebra but has not the extension property). First, we introduce the notion of evolution ideal.

Recall that a subspace $I$ of a commutative algebra $A$ is said to be an \emph{ideal}
if $IA \subseteq I$. While in \cite{Tian} every evolution subalgebra is an ideal, this is not
the case with the definition of ideal given in Definitions \ref{subal} as the following example shows.

\begin{example}\label{dife}
\rm
Let $A$ be an evolution algebra with natural basis $B=\{e_{1,}e_{2},e_{3}\}$ and multiplication given by $e_{1}^{2}=e_{2},$ $e_{2}^{2}=e_{1}$ and $e_{3}^{2}=e_{3}$.
Then, the subalgebra $A'$ generated by $e_{1}+e_{2}$ and $e_{3}$ is an evolution subalgebra with natural basis $B'=\{e_{1}+e_{2},e_{3}\}$  but it is not an ideal as $e_{1}(e_{1}+e_{2})\notin A'$.
\end{example}

On the other hand, not every ideal of an evolution algebra has a natural basis.

\begin{example} \label{evolu-ideal}
\rm
Let $A$ be the evolution algebra with natural basis $B=\{e_{1},e_{2},e_{3}\}$ and product given by $e_{1}^{2}=e_{2}+e_{3}$, $e_{2}^{2}=e_{1}+e_{2}$ and $e_{3}^{2}=-(e_{1}+e_{2})$. Define $u_{1}:=e_{1}^{2}$ and $u_{2}:=e_{2}^{2}$. It is easy to check that $u_{1}$ and
$u_{2}$ are linearly independent and that the subspace they generate is
$$
I:=\{\alpha e_{1}+(\alpha +\beta )e_{2}+\beta e_{3} \ \vert \ \alpha ,\beta \in
\mathbb{K}\}.
$$

Since $e_{1}u_{1}=0$, $e_{2}u_{1}=u_{2}$, $e_{3}u_{1}=-u_{2}$, $e_{1}u_{2}=u_{1}$, $e_{2}u_{2}=u_{2}$
and $e_{3}u_{2}=0$, we have that $I$ is an ideal of $A$. Nevertheless, $I$ has not a natural basis because if $v_{1}$ and $v_{2}$ are elements of $I$ such that $v_{1}v_{2}=0$, then $v_{1}$ and $v_{2}$ are not linearly independent. Indeed, if $v_{1}=\alpha e_{1}+(\alpha +\beta
)e_{2}+\beta e_{3}$ and $v_{2}=\lambda e_{1}+(\lambda+\mu)e_{2}+\mu e_{3}$, for $\alpha,\beta ,\lambda, \mu\in \mathbb{K}$, then
$$
v_{1}v_{2}=\alpha \lambda u_{1}+[(\alpha +\beta )(\lambda+\mu)-\beta\mu]u_{2}.
$$

\noindent
Consequently, if $v_{1}v_{2}=0$ then $\alpha \lambda=0$ and $(\alpha +\beta)(\lambda+\mu)=\beta \mu$. It follows that $\alpha =\lambda=0$, or $\alpha =\beta =0$, or $\lambda=\mu=0$, and hence $v_{1}$ and $v_{2}$ are not linearly independent.
\end{example}

This justifies the introduction of the following definition.

\begin{definition}\label{evoide}
\rm An \emph{evolution ideal} of an evolution algebra $A$ is
an ideal $I$ of $A$ such that $I$ has a natural basis.
\end{definition}

\begin{remark}
\rm

Biologically, an ideal $I$ of an evolution algebra $A$ is a subalgebra such that the reproduction (multiplication) of genotypes of $A$ by single individuals of $I$ produces single individuals in $I$.
\end{remark}

Clearly, evolution ideals are evolution subalgebras but the converse is not true as Example \ref{dife} proves (because an evolution subalgebra does not need to be an ideal). Also
Example \ref{evolu-ideal} shows that an ideal of an evolution
algebra does not need to be an evolution ideal.

\begin{remark}
\rm
In \cite[Definition 4, p. 23]{Tian}, the evolution ideals of an evolution algebra $A$ are defined as
those ideals $I$ of $A$ having a natural basis that can be extended to a natural basis of $A$. It is shown in \cite[Proposition 2, p. 24]{Tian} that every evolution subalgebra is an evolution ideal (in the sense of \cite{Tian}), that is, evolution ideals and evolution subalgebras are the same mathematical concept. This contrasts with our approach (Definitions \ref{subal} and \ref{evoide}, and Examples \ref{dife} and \ref{hui}).
\end{remark}

We finally show that there are examples of evolution ideals for which  no  natural basis can  be extended to a natural basis of the whole evolution algebra. In other words, our definition of evolution ideal is more general than the corresponding definition given in \cite{Tian}.

\begin{example}\label{hui}
\rm
Let $A$ be an evolution algebra with natural basis $B=\{e_{1,}e_{2},e_{3}\}$ and multiplication given by $e_{1}^{2}=e_{3}$, $e_{2}^{2}=e_{1}+e_{2}$ and $e_{3}^{2}=e_{3}$. Let $I$ be the ideal generated by $e_{1}+e_{2}$ and $e_{3}$. Then $I$ is an evolution
ideal with natural basis $B_{0}=\{e_{1}+e_{2},e_{3}\}$.
However no natural basis of $I$ can be extended to a natural basis of $A$. Indeed, if $u=\alpha(e_1+e_2)+\beta e_3$ $v=\gamma(e_1+e_2)+\delta e_3$ and $w=\lambda e_{1}+\mu e_{2}+\rho e_{3}$ is such that the set $\{u,v,w\}$ is a natural basis of $A$, then $uv=0$, $uw=0$ and $vw=0$. This implies the following conditions: $\alpha \gamma =0$, $\beta \delta=0$, $\alpha \mu =0$, $\alpha \lambda + \beta \rho =0$, $\gamma \mu=0$ and $\gamma\lambda+\delta\rho =0$. Therefore, the only possibilities are $\alpha=\delta=\rho=\mu=0$ or $\gamma=\beta=\mu=\lambda=\rho=0$, a contradiction because $\{u,v,w\}$ is a basis. This means that $I$ has not the extension property.
\end{example}

We finish this subsection with the result stating that the class of evolution algebras is closed under quotients by ideals (see also \cite[Lemma 2.9]{EL}). The proof is straightforward.

\begin{lemma}\label{quotient}
Let $A$ be an evolution algebra and $I$ an ideal of $A$. Then $A/I$ with the natural product is an evolution algebra.
\end{lemma}

\begin{remark}
\rm
Let $B:=\{e_{i} \ \vert \ i\in \Lambda \}$ be a natural basis of an
evolution algebra $A$ and let $I$ be an ideal of $A$. Then
$B_{A/I}:=\{\overline{e_{i}}\ \vert \ i\in \Lambda ,e_{i}\notin I\}$ is not necessarily a natural basis of
$A/I$. For an example, consider $A$ and $I$ as in Example \ref{evolu-ideal}. Then $e_1, e_2, e_3\notin I$ and
hence $B_{A/I}=\{\overline{e_1}, \overline{e_2}, \overline{e_3}\}$, which is not a basis of $A/I$ as
the dimension of $A/I$ (as a vector space) is one. Nevertheless, the set $B_{A/I}$ always contains a natural basis of $A/I$.
\end{remark}

Given two  algebras $A$ and $A'$, we recall that a linear map $f: A \to A'$ is said to be an \emph{homomorphism of algebras}  (\emph{homomorphism} for short) if
$f(xy) = f(x) f(y) $ for every $x, y\in A$.

\begin{remark}
\rm
\cite[Theorem 2, p. 25]{Tian} is not valid in general.
Let $I$ be an ideal of an evolution algebra. Then the map $\pi: A \to A/I$ given by $\pi(a)=\overline a$ is a homomorphism of evolution algebras (indeed, $A/I$ is an evolution algebra by Lemma \ref{quotient}) whose kernel is $I$. By \cite[Theorem 2, p. 25]{Tian}, $\Ker(\pi) =I$ is an evolution subalgebra in the sense of \cite{Tian} and, in particular, $I$ has a natural basis. But this  is not always true. For example, take $A$ and $I$ as in Example \ref{evolu-ideal}. Then $I$ is not an evolution ideal (i.e. has not a natural basis), as it is shown in that example.
\end{remark}

We finish by showing that the class of evolution algebras is closed under homomorphic images.

\begin{corollary}
\label{coroquotient}Let $f :A\to A'$ be a homomorphism between the evolution algebras $A$ and $A'$. Then ${\rm Im}(f)$ is an
evolution algebra.
\end{corollary}

\begin{proof} By Lemma \ref{quotient}, $A/\Ker (f) $ is an evolution algebra.
Apply that $\Ker (f) $ is an ideal of $A$ and ${\rm Im}(f)$ is
isomorphic to the evolution algebra $A/\Ker (f)$.
\end{proof}

\medskip


\subsection{Non-degenerate evolution algebras}

In what follows we study the notion of non-degenerate evolution algebra and introduce a radical for an arbitrary evolution algebra such that the quotient by this ideal is a non-degenerate evolution algebra.

\begin{definition}\label{nodegenerada}
\rm
An evolution algebra $A$ is \emph{non-degenerate} if it
 has a natural basis $B=\{e_{i}\ \vert \ i\in \Lambda \}$ such that
 $e_{i}^{2}\neq 0$  for every $i\in \Lambda.$
\end{definition}

\begin{remark}
\rm

That a genotype $e_i$ in an evolution algebra $A$ satisfies $e_i^2=0$ means biologically that is not able to have descendents. By Corollary \ref{MVM} the evolution algebra $A$ will be non-degenerate if all of its genotypes can reproduce.
\end{remark}

In Corollary \ref{MVM} we will show that non-degeneracy does not depend on the considered natural basis. Our proof will rely on the well-known notion of annihilator.

For a commutative algebra $A$ we define its \emph{annihilator}, denoted by
${\rm ann}(A)$, as
$${\rm ann}(A):=\{x\in A \ \vert \ x A = 0\}.$$

\begin{proposition}\label{degen}
Let $A$ be an evolution algebra and $B=\{e_i \ \vert \ i\in \Lambda\}$ a natural basis. Denote by
$\Lambda_0(B):=\{i\ \in \Lambda\ \vert \ e_i^2 =0\}$.
Then:
\begin{enumerate}[\rm (i)]
 \item ${\rm ann}(A)={\rm lin}\{e_i \in B \ \vert \ i\in \Lambda_0(B)\}$.
 \item  ${\rm ann}(A)=0$ if and only if $\Lambda_0=\emptyset$.
 \item ${\rm ann}(A)$ is an evolution ideal of $A$.
 \item $\vert \Lambda_0(B) \vert = \vert\Lambda_0(B') \vert$ for every natural basis $B'$ of $A$.
\end{enumerate}
\end{proposition}
\begin{proof}
By \cite[Lemma 2.7]{EL} we have that ${\rm ann}(A)={\rm lin}\{e_i \in B \ \vert \ i\in \Lambda_0\}$. This implies (i) and (iii). Item (ii) is obvious from (i) and (iv) follows from the fact that $\vert \Lambda_0(B)\vert= {\rm dim}({\rm ann}(A))$.
\end{proof}

From now on, for simplicity, we will write  $\Lambda_0$ instead of $\Lambda_0(B)$.

\begin{corollary}\label{MVM}
An evolution algebra $A$ is non-degenerate if and only if $\rm{ann}(A)=0$. Consequently, the definition of non-degenerate evolution algebra does not depend on the considered natural basis.
\end{corollary}
\begin{proof}
Since $A$ is non-degenerate if and only if $\Lambda_0=\emptyset$, the result  follows directly from  Proposition \ref{degen} (ii).
\end{proof}

\begin{remark}\label{ContraEjemdegen}
\rm
Let $A$ be an evolution algebra and $B=\{e_i \ \vert \ i\in \Lambda\}$ a natural basis. Denote by
 $\Lambda_1:=\{i\ \in \Lambda\ \vert \ e_i^2 \neq 0\}$.
Then:
\begin{enumerate}
\item[\rm (i)]  $A_1:={\rm lin}\{e_i \in B \ \vert \ i\in \Lambda_1\}$ is not necessarily a subalgebra of $A$.
\item[\rm (ii)]  $A/{\rm ann}(A)$ is not necessarily a non-degenerate evolution algebra.
\end{enumerate}
\end{remark}
Indeed, for an example concerning (i), consider the evolution algebra $A$ with natural basis $\{e_1, e_2\}$ and product given by: $e_1^2=0$, $e_2^2 = e_1 + e_2$. Then ${\rm ann}(A)= {\rm lin}\{e_1\}$ and $A_1= {\rm lin}\{e_2\}$, which is not a subalgebra of $A$ as $e_2^2=e_1+e_2\notin A_1$.

To see (ii), let $A$ be an evolution algebra with natural basis $ B=\{e_1, e_2, e_3, e_4, e_5, e_6\}$ such that
$e_1^2=e_2^2=e_3^2=0$, $e_4^2=e_1+e_2$, $e_5^2=e_2$ and $e_6^2=e_2+e_5$. Then
${\rm ann}(A)= {\rm lin}\{e_1, e_2, e_3\} $ and $A/{\rm ann}(A)$ is an evolution algebra generated (as a vector space) by $\overline{e_4}$, $\overline{e_5}$ and $\overline{e_6}$ and has non-zero annihilator. In fact, ${\rm ann}(A/{\rm ann}(A))= {\rm lin}\{\overline{e_4}, \overline{e_5}\}$.
\medskip

To get non-degenerate evolution algebras, we introduce a radical for an evolution algebra $A$, denoted by ${\rm rad}(A)$, in such a way that ${\rm rad}(A/{\rm rad}(A))=\overline 0$, and so $A/{\rm rad}(A)$ is non-degenerate.

\begin{definition}\label{absorcion}
\rm
Let $I$ be an ideal of an evolution algebra $A$. We will say that $I$ has the \emph{absorption property} if  $xA\subseteq I$ implies $x\in I$.
\end{definition}

\begin{remark}
\rm

Biologically, an ideal $I$ has the absorption property if whenever we consider one single individual $x$ of $A$ such that its descendence produces only individuals inside $I$, then the initial individual $x$ belongs to $I$.
\end{remark}

\begin{example}
\rm
Consider the evolution algebra $A$ with natural basis $\{e_1, e_2, e_3\}$ and product given by: $e_1^2= e_2$, $e_2^2=e_1$ and $e_3^2=e_3$. Let $I$ be the ideal of $A$ with basis $\{e_1, e_2\}$. It is not difficult to see that $I$ has the absorption property.
\end{example}

\begin{lemma}\label{ab}
 An ideal $I$ of an evolution algebra $A$ has the absorption property if and only if ${\rm ann}(A/I)=\overline 0$.
\end{lemma}
\begin{proof}
Assume first that $I$ has the absorption property. Take $\overline a\in {\rm ann}(A/I)$. Then $\overline a \ A/I = \overline 0$, so that $aA \subseteq I$. This implies $a\in I$, that is, $\overline a = \overline 0$.
For the converse, use that $a A \subseteq I$ implies $\overline a \ A/I= \overline 0$, that is, $\overline a \in {\rm ann}(A/I) = \overline 0$ and hence $a \in I$.
\end{proof}

\begin{lemma}\label{AbsExt}
Let $I$ be a non-zero ideal of an evolution algebra $A$. Denote by $B=\{e_i \ \vert \ i \in \Lambda\}$ a natural basis of $A$.
If $I$ has the absorption property, then there exists $B_1\subseteq B$ such that $B_1$ is a natural basis of $I$. In particular,  $I$ is an evolution ideal and has the extension property.
\end{lemma}
\begin{proof}
By Lemma \ref{quotient}, we have that $A/I$ is an evolution algebra. Let $\Lambda_2\subseteq \Lambda$ be such that $\overline B: =\{\overline e_i \ \vert \ i \in \Lambda_2\}$ is a natural basis of $A/I$. Denote by $\Lambda_1= \Lambda \setminus \Lambda_2$, and let $B'=\{e_i \ \vert\ i \in \Lambda_1\}$. We claim that $B'$ is a natural basis of $I$.

Take $e_i\in B'$. Then ${\overline {e_i}}\ A/I=\overline 0$; this means ${\overline {e_i}} \in {\rm ann}(A/I)$, which is zero by Lemma \ref{ab}. This implies $e_i \in I$. To see that  $I$ is generated by $B'$, take $y\in I$ and write $y = \sum_{i\in \Lambda_1}k_i e_i + \sum_{i\in \Lambda_2}k_i e_i$ for some $k_i \in \mathbb K$. Taking classes in this identity we get
$\overline 0 = \overline y =  \sum_{i\in \Lambda_2}k_i {\overline{e_i}}\in {\rm lin }\ \overline B$. Since $\overline B$ is a basis, all the $k_i$ (with $i\in \Lambda_2$) must be zero, implying $y = \sum_{i\in \Lambda_1}k_i e_i \in {\rm lin }\ B'$.
\end{proof}

\begin{remark}\label{reciAbsExt}
\rm
The converse of Lemma \ref{AbsExt} is not true. If we take the evolution algebra $A$ with natural basis $\{e_1, e_2\}$ and product given by $e_1^2=e_1$ and $e_2^2=e_1$, then $I= \mathbb K e_1$ is an evolution ideal having the extension property but it has not the absorption property because $e_2 A \subseteq I$ and $e_2\notin I$.
\end{remark}

It is not difficult to prove that the intersection of any family of ideals with the absorption property is again an ideal with the absorption property.

\begin{definition}\label{radical}
\rm
We define the \emph{absorption radical} of  an evolution algebra $A$ as the intersection of all the ideals of $A$ having the absorption property. Denote it by ${\rm rad}(A)$. It is clear that the radical is the smallest ideal of $A$ with the absorption property.
\end{definition}

\begin{proposition}\label{Relacion}
Let $A$ be an evolution algebra. Then ${\rm rad}(A)= 0$ if and only if ${\rm ann}(A)=0$ if and only if $A$ is non-degenerate.
\end{proposition}
\begin{proof}
Note that $ {\rm ann}(A) \subseteq {\rm rad}(A)$, hence ${\rm rad}(A)= 0$ implies ${\rm ann}(A)=0$.
On the other hand, if ${\rm ann}(A)=0$, then $0$ is an ideal having the absorption property. This implies ${\rm rad}(A)=0$ as the radical of $A$ is the intersection of all ideals having the absorption property. Finally,  the assertion ${\rm ann}(A)=0$ if and only if $A$ is non-degenerate follows from Corollary \ref{MVM}.
\end{proof}

\begin{corollary} Let $I$ be an ideal of an evolution algebra $A$. Then $I$ has the absorption property if and only if
${\rm rad}(A/I) = \overline 0$. In particular ${\rm rad}(A/{\rm rad}(A)) = \overline 0$, that is, $A/{\rm rad}(A)$ is a non-degenerate evolution algebra.
\end{corollary}
\begin{proof}
By Lemmas \ref{quotient} and  \ref{ab}, and by Proposition \ref{Relacion} it follows
that $I$ has the absorption property if and only if $\mathrm{ann}(A/I)=0$
(and hence $A/I$ is a non-degenerate evolution algebra), equivalently $
\mathrm{rad}(A/I)=\overline{0}$. Since $\mathrm{rad}(A)$ is an ideal with
the absortion property, the particular case about $A/{\rm rad}(A)$  follows immediately.
\end{proof}

\medskip

We recall that an arbitrary algebra $A$ is \emph{semiprime} if there are no non-zero ideals $I$ of $A$ such that $I^2=0$, and is  \emph{nondegenerate} if $a(Aa)=0$ for some $a\in A$ implies $a=0$. Note that this is a different definition than that of non-degenerate (given in Definition \ref{nodegenerada}).
Although these definitions (in spite of the hyphen) can be confused, they appear with those names in the literature, and this is the reason because of which we compare them.

In the associative case, semiprimeness and nondegeneracy are equivalent concepts. We close this subsection by relating non-degenerate evolution algebras (in the meaning of Definition \ref{nodegenerada}) with semiprime and nondegenerate evolution algebras.
In fact, we obtain the following additional information.

\begin{proposition}\label{nilpotenciaIdealesEvolucion}
Let $A$ be an evolution algebra with non-zero product. Consider the following conditions:
\begin{enumerate}[\rm (i)]
\item\label{iv} $A$ is nondegenerate.
\item\label{i} $A$ is semiprime.
\item\label{ii} $A$ has no non trivial evolution ideals of zero square.
\item\label{iii} $A$ is non-degenerate.
\end{enumerate}
Then: {\rm(\ref{iv})} $\Rightarrow$  {\rm(\ref{i})} $\Leftrightarrow$  {\rm(\ref{ii})}  $\Rightarrow$ {\rm(\ref{iii})}.
\end{proposition}
\begin{proof}
 {\rm(\ref{iv})} $\Rightarrow$  {\rm(\ref{i})}  is well-known for any (evolution or not) algebra.

{\rm(\ref{i})} $\Rightarrow$  {\rm(\ref{ii})} is a tautology.

{\rm(\ref{ii})}  $\Rightarrow$ {\rm(\ref{i})} follows because every ideal $I$ such that $I^2=0$ is an evolution ideal.

{\rm(\ref{ii})} $\Rightarrow$  {\rm(\ref{iii})}. By Proposition \ref{degen}, the annihilator of $A$ is an evolution ideal. Since it has zero square, by the hypothesis, it must be zero. By Proposition \ref{Relacion}
\end{proof}

\begin{remark}
\rm
The implications {\rm(\ref{i})} $\Rightarrow$  {\rm(\ref{iv})} and {\rm(\ref{iii})} $\Rightarrow$  {\rm(\ref{ii})} in Proposition \ref{nilpotenciaIdealesEvolucion} do not hold in general.

To see that  {\rm(\ref{i})} $\not\Rightarrow$  {\rm(\ref{iv})}, consider the evolution algebra $A$ with natural basis $\{e_1, e_2\}$ and product given by $e_1^2 =e_2$ and $e_2^2=e_1+e_2$. Note that $e_1(Ae_1)=0$. Suppose that $I$ is a non-zero ideal such that $I^2=0$. Then it has to be proper and one dimensional because the dimension of $A$ is 2. Therefore $I$ has to be generated (as a vector space) by one element, say, $u=\alpha e_1+ \beta e_2$ for some $\alpha, \beta \in \mathbb K$. Since $0=u^2=\alpha^2 e_2 + \beta^2 (e_1 + e_2)= \beta^2 e_1 +  (\alpha^2 + \beta^2) e_2$ it follows that $\alpha = \beta = 0$, a contradiction.

 To show that {\rm(\ref{iii})} $\not\Rightarrow$  {\rm(\ref{ii})}, let $A$ be the evolution algebra with natural basis $B=\{e_{1},e_{2},e_{3}\}$
and product given by $e_{1}^{2}=e_{2}+e_{3}=e_{2}^{2}$ and $e_{3}^{2}=-e_{2}-e_{3}$.
Then the ideal $I$
 generated by $e_{2}+e_{3}$ is such that $I^{2}=0$ and nevertheless $A$ is non-degenerate.
\end{remark}

\subsection{The graph associated to an evolution algebra}\label{graphAsoc}

We conclude this section by associating a graph to every evolution algebra after fixing a natural basis. This will be very useful because it will allow to visualize when an evolution algebra is reducible or not as well as the results in Subsection \ref{Descomposicion} to get the optimal direct-sum decomposition.

A \emph{directed graph} is a 4-tuple $E=(E^0, E^1, r_E, s_E)$ consisting of two disjoint sets $E^0$, $E^1$ and two maps
$r_E, s_E: E^1 \to E^0$. The elements of $E^0$ are called the \emph{vertices} of $E$ and the elements of $E^1$ the edges of $E$ while for
$f\in E^1$ the vertices $r_E(f)$ and $s_E(f)$ are called the \emph{range} and the \emph{source} of $f$, respectively. If there is no confusion with respect to the graph we are considering, we simply write $r(f)$ and $s(f)$.

 If $s^{-1}(v)$ is a finite set for every $v\in E^0$, then the graph is called \emph{row-finite}. If
$E^0$ is finite and $E$ is row-finite, {then} $E^1$ must necessarily be finite as well; in this case we
say simply that $E$ is \emph{finite}.

\begin{example}\label{ejemploGrafo}
\rm
Consider the following graph $E$:
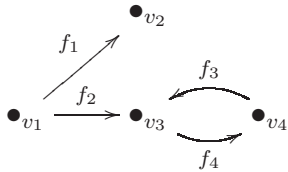
$$
\xymatrix{
 & \bullet_{v_2} & & \\
 \bullet_{v_1}\ar[ur]^{f_1} \ar[r]^{f_2} &  \bullet_{v_3}\ar@/^{-10pt}/ [r]_{f_4}&  \bullet_{v_4}  \ar@/^{-10pt}/ [l]_{f_3}
}
$$
Then $E^0=\{v_1, v_2, v_3, v_4\}$ and $E^1=\{f_1, f_2, f_3, f_4\}$. Examples of source and range are:  $s(f_3)=v_4=r(f_4)$.
\end{example}

A vertex which emits no edges is called a \emph{sink}. A vertex which does not receive any vertex is called a \emph{source}.
A \emph{path} $\mu$ in a graph $E$ is a finite sequence of edges $\mu=f_1\dots f_n$
such that $r(f_i)=s(f_{i+1})$ for $i=1,\dots,n-1$. In this case, $s(\mu):=s(f_1)$ and $r(\mu):=r(f_n)$ are the
\emph{source} and \emph{range} of $\mu$, respectively, and $n$ is the \emph{length} of $\mu$. This fact will be denoted by
$\vert \mu \vert = n$. We also say that
$\mu$ is \emph{a path from $s(f_1)$ to $r(f_n)$} and denote by $\mu^0$ the set of its vertices, i.e.,
$\mu^0:=\{s(f_1),r(f_1),\dots,r(f_n)\}$. On the other hand, by $\mu^1$ we denote the set of edges appearing in $\mu$, i.e., $\mu^1:=\{f_1,\dots, f_n\}$.
We view the elements of $E^{0}$ as paths of length $0$. The set of all paths of a graph $E$ is denoted by ${\rm Path}(E)$.
 Let $\mu = f_1 f_2 \cdots f_n \in {\rm Path}(E)$.
  If  $n = \vert\mu\vert\geq 1$, and if
$v=s(\mu)=r(\mu)$, then $\mu$ is called a \emph{closed path based at $v$}.
 If
$\mu = f_1f_2 \cdots f_n$ is a closed path based at $v$ and $s(f_i)\neq s(f_j)$ for
every $i\neq j$, then $\mu$ is called a \emph{cycle based at} $v$ or simply a \emph{cycle}.

Given a graph $E$ for which every vertex is a finite emitter, the \emph{adjacency matrix} is the matrix $Ad_E=(a_{ij}) \in {\mathbb Z}^{(E^0\times E^0)}$
given by $a_{ij} = \vert\{\text{edges from }i\text{ to }j\}\vert$.

A graph $E$ is said to satisfy  \emph{Condition} (Sing) if among two vertices of $E^0$ there is at most one edge.

There are different ways in which a graph can be associated to an evolution algebra. For instance, we could have considered  weighted evolution graphs (these are  graphs for which
 every edge has associated a weight $\omega _{ij}$, determined by the corresponding structure constant).
 In this way every evolution algebra (jointly with a fixed natural basis) has associated a unique weighted graph, and viceversa. However,  for our purposes we don't need to pay attention to the weights; we only need to take into account if two vertices are connected or not (and in which direction). This is the reason because of which, in order to simplify our approach, it is enough to consider graphs as we do in the following definition.

\begin{definition}\label{EvolGrafo}
\rm
Let $ B=\{e_i\ \vert \ i\in \Lambda\}$ be a natural basis of an evolution algebra $A$ and $M_B=(\omega_{ji})\in  {\rm CFM}_\Lambda(\mathbb K)$ be its structure matrix. Consider the matrix
$P^t=(p_{ji})\in  {\rm CFM}_\Lambda(\mathbb K)$ such that $p_{ji}=0$ if $\omega_{ji}=0$ and $p_{ji}=1$ if $\omega_{ji}\neq 0$.
The \emph{graph associated to the evolution algebra} $A$ (relative to the basis $B$), denoted by $E_A^B$ (or simply by $E$ if the algebra $A$ and the basis $B$ are understood) is the graph whose adjacency matrix is $P= (p_{ij})$.
\end{definition}

Note that the graph associated to an evolution algebra depends on the selected basis. In order to simplify the notation, and if there is no confusion, we will avoid to refer to such a basis.

\begin{example}\label{cambio}
\rm
Let $A$ be the evolution algebra with natural basis $B=\{e_1, e_2\}$ and product given by $e_1^2=e_1+e_2$ and $e_2^2=0$. Consider the natural basis $B'=\{e_1+e_2, e_2\}$. Then the graphs associated to the bases $B$ and $B'$ are, respectively:
$$
E: \quad 
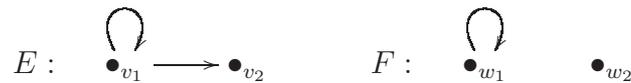
\xymatrix{
\bullet_{v_1}  \uloopr{}  \ar[r] &  \bullet_{v_2}
}
\quad \quad \quad
F:\quad \xymatrix{
\bullet_{w_1}  \uloopr{}  &  \bullet_{w_2}
}
$$

\end{example}

\begin{example}\label{primerEjemplo}
\rm
Let $A$ be the evolution algebra with natural basis $B=\{e_1, e_2, e_3, e_4\}$ and product given by: $e_1^2=e_2+e_3$, $e_2^2=0$, $e_3^2=-2e_4$ and $e_4^2=5e_3$. Then the adjacency matrix of the graph associated to the basis $B$ is:

$$P=
\left(
\begin{matrix}
0 & 1 & 1 & 0 \cr
0 & 0 & 0 & 0 \cr
0 & 0 & 0 & 1 \cr
0 & 0 & 1 & 0 \cr
\end{matrix}
\right)
$$
and $E$ is the graph given in Example \ref{ejemploGrafo}.

\end{example}
\medskip

Now, conversely, to every row-finite graph satisfying Condition (Sing)
we associate an evolution algebra whose corresponding structure matrix consists of 0 and 1, as follows.

\begin{definition}
\rm
Let $E$ be a row-finite graph satisfying Condition (Sing)
and  $P=(p_{ij})$ be its adjacency matrix. Assume $E^0 =\{v_i\}_{i\in \Lambda}$. For every field $\mathbb K$
the \emph{evolution}  $\mathbb K$-\emph{algebra associated to the graph} $E$, denoted by $A_E$, is the free  algebra whose underlined vector space has a natural basis $B=\{e_i\}_{i\in \Lambda}$ and with structure matrix relative to $B$ given by $P^t=(p_{ji})$.
\end{definition}

\begin{example}\label{EjemploA}
\rm
Let $E$ be the following graph:

$${
\tiny
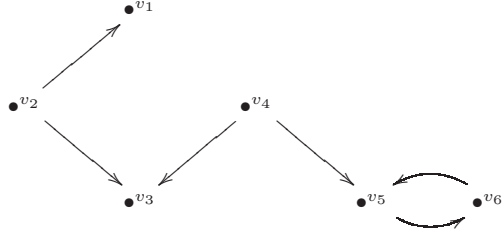
\xymatrix{
                                       &\bullet^{v_1} &                            & &   \\
\bullet^{v_2} \ar[ur]\ar[dr]&                     & \bullet^{v_4} \ar[dl] \ar[dr]& & \\
                                        &\bullet^{v_3} &                                         &\bullet^{v_5}  \ar@/^{-10pt}/ [r]& \bullet^{v_6}   \ar@/^{-10pt}/ [l]
}}$$
\vspace{.2truecm}
Its adjacency matrix is
$$P=
\left(
\begin{matrix}
0 & 0 & 0 & 0 & 0 & 0\cr
1 & 0 & 1 & 0 & 0 & 0\cr
0 & 0 & 0 & 0 & 0 & 0\cr
0 & 0 & 1 & 0 & 1 & 0\cr
0 & 0 & 0 & 0 & 0 & 1\cr
0 & 0 & 0 & 0 & 1 & 0
\end{matrix}
\right)$$

\noindent
and the corresponding evolution algebra is the algebra $A$ having a natural basis $B=\{e_1, \dots, e_6\}$ and product  determined by: $e_1^2= 0$, $e_2^2=e_1+e_3$, $e_3^2=0$, $e_4^2=e_3+e_5$, $e_5^2=e_6$ and $e_6^2=e_5$.
\end{example}

\begin{remark}\label{sumideros}
\rm
It is easy to determine the annihilator of an evolution algebra $A$ by looking at the sinks of the graph associated to a basis. By Proposition \ref{degen}, the annihilator of $A$ consists of the linear span of the elements of the basis whose square is zero (these are, precisely, the sinks of the corresponding graph).
For instance, in Example \ref{EjemploA}, ${\rm ann}(A) = {\rm lin}\{e_1, e_3\}$.
\end{remark}


\section{Ideals generated by one element}\label{UnElemento}
In order to characterize those ideals generated by one element, we introduce the following useful definitions.

\begin{definitions}
\label{descen}\
\rm
Let $B=\{e_{i}\ \vert \ i\in \Lambda \}$ be a natural basis
of an evolution algebra $A${\ and let }$i_{0}\in \Lambda .${\ The
\emph{first-generation descendents} of }$\ i_{0}$ {\ are the elements
of the subset }$D^{1}(i_{0})${\ given by: }
\begin{equation*}
D^{1}(i_{0}):=\left\{k\in \Lambda \ \vert \ e_{i_{0}}^{2}=\sum_k \omega _{ki_{0}}e_{k}\text{
with }\omega _{ki_{0}}\neq 0\right\}.
\end{equation*}
{In an abbreviated form, }$D^{1}(i_{0}):=\{j\in \Lambda\ \vert\ \omega
_{ji_{0}}\neq 0\}.${\ Note that }$j\in D^{1}(i_{0})${\ if and only
if, }$\pi _{j}(e_{i_{0}}^{2})\neq 0${\ (where }$\pi _{j}${\ is the canonical
projection of }$A${\ over }$\mathbb Ke_{j}${).}

{Similarly, we say that }$j${\ is a \emph{second-generation
descendent} of }$i_{0}${\ whenever }$j\in D^{1}(k)${\ for some }$
k\in D^{1}(i_{0}).${\ Therefore,}
\begin{equation*}
D^{2}(i_{0})=\bigcup\limits_{k\in D^{1}(i_{0})}D^{1}(k).
\end{equation*}
{By recurrency, we define the set of \emph{mth-generation descendents
} of }$i_{0}${\ as}
\begin{equation*}
{\ }D^{m}(i_{0})=\bigcup\limits_{k\in D^{m-1}(i_{0})}D^{1}(k).
\end{equation*}
{\ Finally, the \ s\emph{et of descendents} of }$i_{0}${\ is
defined as the subset of }$\Lambda ${\ given by }
\begin{equation*}
D(i_{0})=\bigcup\limits_{m\in \mathbb{N}}D^{m}(i_{0}).
\end{equation*}
On the other hand, we say that $j\in \Lambda$ is an
\emph{ascendent} of $i_{0}$ if $i_{0}\in D(j);$ that is, $i_{0}$
is a \emph{descendent of} $j.$
\end{definitions}
\medskip

\begin{remark}
\rm

From a biological point of view, the first-generation descendents of (the genotype) $i$ are the genotypes appearing in $e_i^2$ (note that here we are identifying $e_i$ and $i$).

The second-generation descendents of (the genotype) $i$ are the genotypes appearing in the reproduction of the first-generation descendents of  $e_i$.

In general, the mth-generation descendents of (the genotype) $i$ are the genotypes appearing in the reproduction of the (m-1)th-generation descendents of  $e_i$.

The set of descendents of $i$ are the genotypes appearing in the nth-generation descendents of $i$ for an arbitrary generation n.
\end{remark}

We illustrate the definitions just introduced in terms of the underlying graph associated to an evolution algebra (relative to a natural basis). We will abuse of the notation for simplicity.

\begin{definitions}\label{descendientes}
\rm
Let $E$ be a graph. For a vertex $j\in E^0$ we define:
$$D^m(j) :=\{v\in E^0 \ \vert \ \text{ there is a path } \mu \text{ such that } |\mu |= m,  s(\mu)=v_j, r(\mu)=v\}.$$
In words, the elements of $D^m(j)$ are those vertices to which $v_j$ connects via a path of length $m$. We also define
$$D(j) =\bigcup_{m\in \N}D^m(j)=\{v\in E^0 \ \vert \ \text{ there is a path } \mu \text{ such that }  s(\mu)=v_j, r(\mu)=v\}.$$
When we want to emphasize the graph $E$ we will write $D_E^m(j)$ and $D_E(j)$, respectively.
\end{definitions}

\begin{examples}
\rm
Let $E$ and $F$ be the following graphs:

$$
E:  
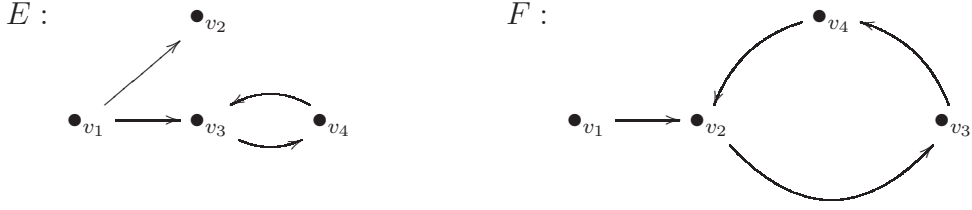
\xymatrix{
 & \bullet_{v_2} & & \\
 \bullet_{v_1}\ar[ur] \ar[r] &  \bullet_{v_3}\ar@/^{-10pt}/ [r]&  \bullet_{v_4}  \ar@/^{-10pt}/ [l]
}
\quad \quad
F:  \xymatrix{
& {} & \bullet_{v_4}  \ar@/^{-10pt}/ [ld] &  {} \\
\bullet_{v_1}\ar[r] & \bullet_{v_2}  \ar@/^{-30pt}/ [rr]&  & \bullet_{v_3} \ar@/^{-10pt}/ [lu] \\
}
$$
\vspace{.70truecm}

Some examples of the sets of the $n$th-generation descendents and of the set of descendents of some indexes are the following.

 $D^1_E(3)=\{v_4\}=D^{1+2m}_E(3)$; $D^2_E(3)=\{v_3\}=D^{2m}_E(3)$ for every $m\in \N$, and so $D_E(3) = D^1_E(3)\cup D^2_E(3)=\{v_3, v_4\}$.

 $D^1_F(2)=\{v_3\}=D^{1+3m}_F(2)$; $D^2_F(2)=\{v_4\}=D^{2+3m}_F(2)$; $D^3_F(2) = \{v_2\}= D^{3m}_F(2)$ for every $m\in \N$, and so $D_E(3) = D^1_F(2)\cup D^2_F(2)\cup D^3_F(2)= \{v_2, v_3, v_4\}$.
\end{examples}

\medskip

Next we characterize the descendents (and hence the ascendents)\ of every
index $i_{0}\in \Lambda.$ More precisely, we describe the set $D^{m}(i_{0}).$

\medskip

\begin{proposition}\label{Desce}
Let $B=\{e_{i}\ \vert \ i\in \Lambda \}$ be a natural basis of an evolution
algebra $A.$ Consider $i_{0},j\in \Lambda $ and $m\geq 2$.
\begin{enumerate}[\rm (i)]
\item If $j\in D^{1}(i_{0})$ (if and only if $\omega _{ji_{0}}\neq 0$), then
\begin{equation*}
e_{j}e_{i_{0}}^{2}=\omega _{ji_{0}}e_{j}^{2}.
\end{equation*}

\item $j\in D^{m}(i_{0})$ if and only if there exist $k_{1},k_{2},\dots
,k_{m-1}\in \Lambda $ such that
$$
\omega _{jk_{m-1}}\omega _{k_{m-1}k_{m-2}}\cdots\omega _{k_{2}k_{1}}\omega
_{k_{1}i_{0}}\neq 0,
$$
in which case,
$$
e_{j}^{2}=\left({\omega_{jk_{m-1}}}\omega _{k_{m-1}k_{m-2}}\cdots\omega_{k_{2}k_{1}}\omega _{k_{1}i_{0}}\right)^{-1}
e_{j}e_{k_{m-1}}e_{k_{m-2}} \dots e_{k_{2}}e_{k_{1}}e_{i_{0}}^{2}.
$$
\end{enumerate}
\end{proposition}

\begin{proof}
Note that $j\in D^{1}(i_{0})\,$if and only if $\omega _{ji_{0}}\neq 0$, in
which case $e_{j}e_{i_{0}}^{2}=\omega _{ji_{0}}e_{j}^{2}$ so that
$$
e_{j}^{2}=\omega _{ji_{0}}^{-1}e_{j}e_{i_{0}}^{2}.
$$
Suppose that the result holds for $m-1.$ Thus, $k\in D^{m-1}(i_{0})\,$\ if
and only if there exist $k_{1},k_{2},\dots ,k_{m-2}\in \Lambda $ such that
$$
e_{k}^{2}=({\omega _{kk_{m-2}}\omega
_{k_{m-2}k_{m-3}}\cdots \omega _{k_{2}k_{1}}\omega _{k_{1}i_{0}}})^{-1}
e_{k}e_{k_{m-2}}\cdots e_{k_{2}}e_{k_{1}}e_{i_{0}}^{2}.
$$

\noindent Let $j\in D^{m}(i_{0}).$ This means that $j\in D^{1}(k)$ for some $k\in
D^{m-1}(i_{0})$, so that $\omega _{jk}\neq 0,\,\ $and hence
$e_{j}^{2}=({\omega _{jk}})^{-1}e_{j}e_{k}^{2}.$ Consequently,
$$
e_{j}^{2}=({\omega _{jk}\omega
_{kk_{m-2}}\cdots \omega _{k_{2}k_{1}}\omega
_{k_{1}i_{0}}})^{-1}e_{j}e_{k}e_{k_{m-2}}\cdots e_{k_{2}}e_{k_{1}}e_{i_{0}}^{2},
$$
as desired.
\end{proof}

\medskip

From Proposition \ref{Desce} we deduce that if $i$ is a descendent
of $j$, and if $j$ is a descendent of $k$, then $i$ is a descendent of $k.$

Another direct consequence of the mentioned proposition is the corollary that follows.
From now on, if $S$ is a subset of an algebra $A$ then we will denote by
$\left\langle S\right\rangle $ the ideal of $A$ generated by $S.$

\begin{corollary}
Let $B=\{e_{i}\ \vert \ i\in \Lambda \}$ be a natural basis of an evolution algebra $A.$
If $j\in \Lambda \,$\ is a descendent of $i_{0}\in \Lambda,$ then $
\left\langle e_{j}^{2}\right\rangle \subseteq \left\langle
e_{i_{0}}^{2}\right\rangle$.
\end{corollary}

\medskip

Proposition \ref{Desce}  will allow to describe easily the ideal
generated by an element in a natural basis, as well as the ideal generated by its square.

\begin{corollary}\label{nice}
Let $A$ be an evolution algebra and $B=\{e_{i}\ \vert \ i\in \Lambda \}$
a natural basis. Then, for every $k\in \Lambda,$
$$
\left\langle e_{k}^{2}\right\rangle ={\rm lin}\{e_{j}^{2}\ \vert \ j\in D(k)\cup
\{k\}\} \quad \hbox{and} \quad  \left\langle e_{k}\right\rangle =\mathbb{K} e_{k}+\left\langle e_{k}^{2}\right\rangle.
$$
\end{corollary}

\begin{proof}
Since $D^{1}(k)=\{j\in \Lambda \ \vert \ \omega _{jk}\neq 0\}$, by
Proposition \ref{Desce} we have
$$Ae_{k}^{2}={\rm lin}\{e_{j}^{2}\ \vert \ j\in D^{1}(k)\}.$$ Consequently,
$A(Ae_{k}^{2})={\rm lin}\{e_{j}^{2}\ \vert \ j\in D^{2}(k)\},$ and, therefore,
$
\left\langle e_{k}^{2}\right\rangle ={\rm lin}\{e_{j}^{2}\ \vert \ j\in D(k)\cup \{k\}\}.
$ The rest is clear.
\end{proof}

Another proof of Corollary \ref{nice}  will be obtained in Proposition \ref{cuadrado}.

\begin{remark}\label{numerabilidad}
\rm
Since $ \left\langle e_{k}\right\rangle =\mathbb{K} e_{k}+\left\langle e_{k}^{2}\right\rangle$, it is clear that $\left\langle e_{k}\right\rangle
=\left\langle e_{k}^{2}\right\rangle $ if and only if $e_{k}\in \left\langle
e_{k}^{2}\right\rangle$.
On the other hand, because $D(k)$ is at most countable, by definition, the dimension of $\left\langle e_{k}\right\rangle $ is, at most, countable.
\end{remark}
\medskip

We can also describe the ideal generated by any element in a natural basis of an evolution algebra in terms of multiplication operators. This result will be very useful in order to characterize simple evolution algebras.

\begin{definitions}\label{multipli}
\rm
Let $A$ be an evolution $\mathbb K$-algebra. For any element $a\in A$ we define the multiplication operator by $a$, denoted by $\mu_a$, as the following map:
$$
\begin{matrix}
\mu_a: & A & \to A\\
& x & \mapsto ax
\end{matrix}
$$
By $\mu_A$ we will mean the linear span of the set $\{\mu_a \ \vert \ a \in A\}$.
For an arbitrary $n \in \N$, denote by $\mu_A^n$:
$$
\mu_A^n:={\rm lin}\{\mu_{a_1}\dots \mu_{a_n}\ \vert \ a_1, \dots, a_n\in A\}.
$$
For $n=0$ we define $\mu_a^0$ as the identity map $i_A: A \to A$, while $\mu_A^0$ denotes $\mathbb K i_A$.
Now, for $x\in A$, the notation $\mu_A^n(x)$ will stand for the following linear span:
\begin{eqnarray*}
\mu_A^n(x):& = &{\rm lin}\{\mu_{a_1}\mu_{a_2}\dots \mu_{a_{n-1}}\mu_{a_n}(x)\ \vert \ a_1, \dots, a_n \in A\} \\
& = & {\rm lin}\{a_1(a_2(\dots (a_{n-1}(a_n x)) \dots)\ \vert \ a_1, \dots, a_n \in A\}.
\end{eqnarray*}
For example, $\mu_A^3(x) ={\rm lin}\{a_1(a_2 (a_3 x)))\ \vert \ a_1, a_2, a_3 \in A\}$.
\end{definitions}

\begin{definition}\label{lambda}
\rm
Let $A$ be an evolution algebra with a natural basis $B=\{e_{i}\ \vert \ i\in \Lambda \}$.
For any $x\in A$, we define
$$\Lambda^x :=\{i\in \Lambda \ \vert \ e_ix\neq 0\}.$$
\end{definition}

\begin{proposition}\label{cuadrado}
Let $A$ be an evolution algebra with a natural basis $B=\{e_{i}\ \vert \ i\in \Lambda \}$.
\begin{enumerate}[\rm (i)]
\item  Let $k\in \Lambda$ be such that $e_k^2\neq 0$.
\medskip
\begin{enumerate}[\rm (a)]
\item $\mu_A^n (e_k^2)= {\rm lin}\{e_j^2 \ \vert \ j\in D^n(k)\},$ for every $n\in \N$.
\item $ \left\langle e_{k}^{2} \right\rangle = {\rm lin}\bigcup\limits_{n=0}^\infty \mu_A^n(e_k^2). $
\medskip
\item $\left\langle e_k^2\right\rangle = {\rm lin}\{e_j^2\ \vert \ j\in D(k)\cup \{k\}\}.$
 \end{enumerate}
 \medskip
\item For any $x\in A$,
\medskip
\begin{enumerate}[\rm (a)]
\item $\mu^1_A(x)= {\rm lin}\{e_i^2 \ \vert \ i\in \Lambda^x\}$ and
for any $n\geq 2$, $\mu_A^n(x)= {\rm lin}\bigcup\limits_{i\in \Lambda^x}\{e_j^2\ \vert \ j\in D^{n-1}(i)\}$.
\item $\left\langle x\right\rangle = {\rm lin} \bigcup\limits_{n=0}^{\infty}\mu_A^n(x).$
\end{enumerate}
\end{enumerate}
\end{proposition}
\begin{proof}
We prove (a) in item (i) by induction. Suppose first $n=1$. Note that
$e^2_k = \sum\limits_{i\in D^1(k)}\omega_{ik}e_{i}$ with $\omega_{ik}\in \mathbb K\setminus\{0\}$.
For an arbitrary $e_l\in B$ we have
$e^2_k e_l = \sum_{i\in D^1(k)}{\omega_{ik}}e_{i}e_l.$
This sum is zero, if $l\neq i$ for every $i\in D^1(k)$, or it coincides with $\omega_{ik}e_{i}^2$ if $l= i$ for some $i$. Therefore,
$$\mu_A^1(e_{k}^{2})\subseteq {\rm lin}\{e^2_{i}\ \vert \ i\in D^1(k)\}.$$
To show
${\rm lin}\{e^2_{i}\ \vert \ i\in D^1(k)\}\subseteq  \mu_A^1(e_{k}^{2})$, take any
$e_{i}$ with $i\in D^1(k)$. By Proposition \ref{Desce} (i) we have
$e_i^2 =  {\omega_{ik}}^{-1}e_k^2e_{i}\subseteq  \mu_A^1(e_k^2)$. This finishes the first step in the induction process.

Assume we have the result for $n-1$. Using the induction hypothesis we get:
\begin{eqnarray*}
\mu_A^n(e_k^2) & = & A\ \mu_A^{n-1}(e_k^2)= A\ \left({\rm lin}\{e_i^2\ \vert \ i\in D^{n-1}(k)\}\right) =
{\rm lin}\bigcup\limits_{i\in D^{n-1}(k)}\mu_A^1(e_i^2) \\
& = & {\rm lin} \bigcup\limits_{i\in D^{n-1}(k)}\{e_j^2 \ \vert \ j\in D^1(i)\} = {\rm lin}\{e_j^2 \ \vert \ j\in D^n(k)\}.
\end{eqnarray*}
This proves (a) in (i). Item (b) in (i) follows immediately from (a) and item (c) can be obtained from (a) and (b).

Now we prove (ii).
Note that $\mu^1_A(x)= {\rm lin}\{e_i^2 \ \vert \ i\in \Lambda^x\}$. It is not difficult to see that, for $n>1$,
$$\mu_A^n(x) = {\rm lin}\bigcup\limits_{i\in \Lambda^x}\mu_A^{n-1}(e_i^2).$$
Apply condition (b) in item (i) to finish the proof of (a) in (ii).
Finally, item (b) in (ii) is  easy to check.
\end{proof}

\begin{corollary}\label{dimensGeneral}
Let $A$ be an evolution algebra. Then for any element $x\in A$ the dimension of the ideal generated by $x$ is at most countable.
\end{corollary}
\begin{proof} By (ii) in Proposition \ref{cuadrado} the dimension of the ideal generated by $x$ is the dimension of $\cup_{n=0}^\infty \mu_A^n(x)$. Since any $\mu_A^n(x)$ is finite dimensional, for every $n\in \N \cup \{0\}$, we are done.
\end{proof}


\section{Simple evolution algebras}\label{Simple}
This section is addressed to the study and characterization of simple evolution algebras.
We recall that an algebra $A$ is \emph{simple} if $A^2 \neq 0$ and $0$ is the only proper ideal.

\begin{proposition}\label{simple}
Let $A$ be an evolution algebra and let $B=\{e_{i}\ \vert \ i\in \Lambda \}$
be a natural basis of $A$.$\,$ Consider the following conditions:

\begin{enumerate}[\rm (i)]
\item\label{uno} $A$ is simple.
\item\label{dos} $A$
satisfies the following properties:
\begin{enumerate}[\rm (a)]
\item $A$ is non-degenerate.
\item $A={\rm lin}\{e_{i}^{2}\ \vert \ i\in \Lambda \}$.
\item If ${\rm lin}\{e_{i}^{2}\ \vert \ i\in \Lambda'\}$ is a non-zero ideal of $A$ for a
non-empty $\Lambda'\subseteq \Lambda $ then $\vert \Lambda'\vert =\vert \Lambda \vert.$
\end{enumerate}
\end{enumerate}

Then: {\rm (i)} $\Rightarrow$ {\rm (ii)} and {\rm (ii)} $\Rightarrow$ {\rm (i)} if $\vert \Lambda \vert < \infty$.
Moreover, if $A$ is a simple evolution algebra, then the dimension of $A$ is at most countable.
\end{proposition}

\begin{proof}
(i) $\Rightarrow$ (ii).
Suppose first that $A$ is a simple evolution algebra. If $A$ is degenerate, then $e^2=0$ for some
element $e$ in a natural basis $B$ of $A$. Then ${\rm lin}\{e\}$ is a nonzero ideal of $A$.
The simplicity implies ${\rm lin}\{e\}=A$, but then $A^2=0$, a contradiction. This shows (a).

Note that $A^2 = {\rm lin}\{e_{i}^{2}\ \vert \ i\in \Lambda \}$ is an ideal of $A$. Since $A^2 \neq 0$ and $A$ is simple, we have $A=A^2$, which is (b).

If ${\rm lin}\{e_{i}^{2}\ \vert \ i\in \Lambda'\}$ is a non-zero ideal of $A$, the simplicity of $A$ implies
${\rm lin}\{e_{i}^{2}\ \vert \ i\in \Lambda'\}=A=  {\rm lin}\{e_{i}^{2}\ \vert \ i\in \Lambda \}
= {\rm lin}\{e_{i}\ \vert \ i\in \Lambda \}$.  This gives $\vert \Lambda'\vert =\vert \Lambda \vert.$
\medskip

(ii) $\Rightarrow$ (i). Assume that the dimension of $A$ is finite, say $n$.
Since $A$ satisfies (a), $A^2\neq 0$. To prove that $A$ is simple, suppose that this is not the case. Then, there exists $u\in A$ such that $\left\langle u\right\rangle $ is a non-zero proper ideal of $A.$ Let $k\in \Lambda $ be such that $\pi _{k}(u)\neq 0.$ Then $ \left\langle e_{k}^{2} \right\rangle$ is a
non-zero ideal of $A$ contained in $ \left\langle u\right\rangle$, so that
$\left\langle e_{k}^{2}\right\rangle $ is proper.
Proposition \ref{cuadrado}
implies that $\left\langle e_{k}^{2} \right\rangle = {\rm lin}\{e^2_j\ \vert \ j \in D(k)\cup \{k\}\}$, which, by (a), is a non-zero ideal of $A$. Use (c) to get
$\vert \Lambda\vert = \vert D(k) \cup \{k\}\vert $. Since $D(k) \cup \{k\} \subseteq \Lambda$, we have
$ \Lambda =  D(k) \cup \{k\}$.
 Now using (b),
$$\left\langle e_{k}^{2} \right\rangle = {\rm lin}\{e_j^2\ \vert \ j \in D(k) \cup \{k\}\}=
{\rm lin}\{e^2_j\ \vert \ j\in \Lambda\}=A,$$
a contradiction as $\left\langle e_{k}^{2} \right\rangle$ is a proper ideal of $A$.
\medskip

The dimension of $A$ is at most countable when $A$ is simple by Corollary \ref{dimensGeneral}.
\end{proof}

Although every simple evolution algebra is non-degenerate, at most countable dimensional and coincides with the linear span of the square of the elements of any natural basis, as Proposition \ref{simple} says, the converse is not true because  the hypothesis of finite dimension
is necessary as the following example shows.

\begin{example}\label{infinito}
\rm Let $A$ be an evolution algebra with natural basis $\{e_i \ \vert \ i \in \N\}$ and product given by:

\begin{displaymath}
{
\begin{array}{lll}
e_1^2  &=&e_3+e_5     \\
e_3^2&=&e_1+e_3+e_5\\
e_5^2&=&e_5+e_7    \\
e_7^2&=&e_3+e_5+e_7      \\
 & \vdots &
\end{array}
}
\qquad
{
\begin{array}{lll}
e_2^2&=&e_4+e_6\\
        e_4^2&=&e_2+e_4+e_6\\
           e_6^2&=&e_6+e_8\\
      e_8^2&=&e_4+e_6+e_8\\
 & \vdots &
      \end{array}
      }
\end{displaymath}

Then $A$ satisfies the conditions (a), (b) and (c) in  Proposition \ref{simple} (ii)  but $A$ is not simple as
$\left\langle e_1^2\right\rangle$ and $\left\langle e_2^2\right\rangle$ are two nonzero proper ideals.
\end{example}

\begin{example}\label{noSimple}
\rm
Consider the evolution algebra $A$ having a natural basis $\{e_1, e_2\}$ and product given by $e_i^2=e_i$ for $i=1, 2$. Then $\left\langle e_i\right\rangle= \mathbb K e_i$ is a non-zero proper ideal of $A$. This means that the condition (c)  in Proposition \ref{simple} (ii) cannot be dropped.
\end{example}

We show now that there exist simple evolution algebras of infinite dimension.

\begin{example}
\rm
Let $A$ be the evolution algebra with natural basis $\{e_i \ \vert \ i \in \N\}$ and product given by:
\begin{eqnarray*}
e_{2n-1}^2 =& e_{n+1}+e_{n+2}\\
 e_{2n}^2 =& e_n+e_{n+1}+e_{n+2}
 \end{eqnarray*}
 Then $A$ is simple.
\end{example}

\begin{remark}
\rm
An evolution algebra $A$ whose associated graph (relative to a natural basis) has sinks cannot be simple. The reason is that a sink corresponds to an element in a natural basis of zero square, hence to an element in the annihilator of $A$. By Proposition \ref{simple}, every simple evolution algebra has to be non-degenerate.
\end{remark}
\medskip

\begin{corollary}\label{simple-mat}
Let $A$ be a finite-dimensional evolution algebra of dimension $n$ and $B= \{e_i \ \vert \ i \in \Lambda\}$ a natural basis of $A$. Then $A$ is simple if and only if the determinant of the structure matrix $M_B(A)$ is non-zero and $B$ cannot be reordered in such a way that the corresponding structure matrix is as follows:
\begin{equation*}
 M':=\left(
\begin{array}{cc}
W_{m\times m} & U_{m\times (n-m)} \\
0_{(n-m)\times m} & Y_{(n-m)\times (n-m)}
\end{array}
\right),
\end{equation*}
for some $m\in \mathbb N$ with $m<n$ and matrices $W_{m\times m},$\textit{\ }$U_{m\times
(n-m)}$\textit{\ and }$Y_{(n-m)\times (n-m)}.$
\end{corollary}
\begin{proof}
If $A$ is simple then, by Proposition \ref{simple}, $A={\rm lin}\{e_{i}^{2}\ \vert \ i\in \Lambda \}$.
This means that the determinant of $M_B(A)$ is non-zero. To see the other condition, take into account that a reordering of the basis $B$ producing a matrix as $M'$ would imply that $A$ has a proper ideal of dimension $m\geq 1$, a contradiction as we are assuming that $A$ is simple.

Conversely, if $\vert M_B(A) \vert \neq 0$, then $A$ is generated by the linear span of $\{e^2 \ \vert \ e\in B\}$. On the other hand, $A$ cannot be degenerate as, otherwise, ${\rm ann}(A) = {\rm lin}\{e \in B\ \vert \ e^2=0\}$ (see Proposition \ref{degen}). Decompose $B$ as $B= B_0 \sqcup B_1$, where $B_0=\{e \in B\ \vert \ e^2=0\}$ and $B_1 = B \setminus B_0$ and let $B'$ be a reordering of $B$ in such a way that the first elements correspond to the elements of $B_0$ and the rest to the elements of $B_1$. Then $M_{B'}(A)$ is as matrix $M'$ in the statement, a contradiction.
We have shown that $A$ satisfies conditions (a) and (b) in Proposition \ref{simple} (ii). Now we see that condition (c) is also satisfied. Assume that $\Lambda'\subseteq \Lambda$ is such that ${\rm lin}\{e_i\ \vert \ i\in \Lambda'\}$ is a non-zero ideal of $A$. If we reorder $B$ in such a way that the first elements are in $\{e_i\ \vert \ i\in \Lambda'\}$, then the corresponding structure matrix is as  $M'$ in the statement, a contradiction. Now use Proposition \ref{simple} to prove that $A$ is simple.
\end{proof}

Next we characterize simple evolution algebras of arbitrary dimension.

\begin{theorem}\label{CharacSimple}
Let $A$ be a non-zero evolution algebra and $B=\{e_i\ \vert\ i \in \Lambda\}$ a natural basis. The following conditions are equivalent.
\begin{enumerate} [\rm (i)]
\item\label{algo1} $A$ is simple.
\item\label{algo2}  If ${\rm lin}\{e_i^2 \ \vert \ i \in \Lambda'\}$ is an
ideal for a nonempty subset $\Lambda'\subseteq \Lambda$, then $A={\rm lin}\{e_i^2 \ \vert \ i \in \Lambda'\}$.
\item\label{algo3}  $A= \langle e_i^2\rangle = {\rm lin}\{e_j^2 \ \vert \ j\in D(i)\}$ for every $i \in \Lambda$.
\item\label{algo4}  $A= {\rm lin}\{e_i^2 \ \vert \ i \in \Lambda\}$ and  $\Lambda = D(i)$ for every $i \in \Lambda$.
\end{enumerate}
\end{theorem}
\begin{proof}
(\ref{algo1}) $\Rightarrow$ (\ref{algo2}). If $A$ is simple, then it is non-degenerate by Proposition \ref{simple} and hence, $ {\rm lin}\{e_i^2 \ \vert \ i \in \Lambda'\}$ is a nonzero ideal of $A$, so that the result follows.

(\ref{algo2}) $\Rightarrow$ (\ref{algo3}). By Proposition \ref{cuadrado} (ii) we have $ \langle e_i^2\rangle = {\rm lin}\{e_j^2 \ \vert \ j\in D(i)\cup \{i\}\}$. By (\ref{algo2}), this set is $A$. Since $e_i\in A=  \langle e_i^2\rangle$ we have $i \in D(i)$ and (\ref{algo3}) has being proved.

(\ref{algo3}) $\Rightarrow$ (\ref{algo4}). Since $D(i) \subseteq \Lambda$, we have $A= {\rm lin}\{e_i^2 \ \vert \ i \in \Lambda\}.$ Now, take $j\in \Lambda$. Then $e_j\in A= \langle e_i^2\rangle = {\rm lin}\{e_k^2 \ \vert \ k\in D(i)\}$ (by (ii)). It follows that $j\in D(i)$ and therefore $\Lambda \subseteq D(i)$.

(\ref{algo4}) $\Rightarrow$ (\ref{algo1}). Let $I$ be a nonzero ideal of $A$.
Since $Ie_i\neq 0$ for some $i\in \Lambda$, then $e_i^2\in I$ and so $I\supseteq \langle e_i^2 \rangle= A$ (by (\ref{algo4})).
\end{proof}

The two conditions in Theorem \ref{CharacSimple} (\ref{algo4}) are not redundant as we see in the next examples.

\begin{examples}
\rm
Consider the evolution algebra $A$ with natural basis $\{e_1, e_2, e_3\}$ and product given by $e_1^2 = e_3^2 =e_1 + e_2$; $e_2^2 = e_3$. Then $\{1, 2, 3\}= D(i)$ for every $i\in \{1, 2, 3\}$ but ${\rm lin}\{e_i^2\ \vert \ i=1, 2, 3\}= {\rm lin}\{e_1+e_2, e_3\}\neq A$.

On the other hand, consider the evolution algebra $A$ given in Example \ref{infinito}. Then $A= {\rm lin}\{e_i^2 \ \vert \ i\in \N\}$ but $\N \neq D(i)$ for every $i\in \N$.
\end{examples}

Now we show that in Theorem \ref{CharacSimple} (\ref{algo3}) the hypothesis ``for every $i\in \Lambda$'' cannot be eliminated.

\begin{example}
\rm
Let $A$ be the evolution algebra with natural basis $B=\{e_n \ \vert \ n \in \N\}$ and product given by:
\begin{displaymath}
{
\begin{array}{lll}
e_1^2&  =&  0 \\
e_2^2  &=&e_3+e_5     \\
e_4^2&=&e_1+e_3+e_5\\
e_6^2&=&e_5+e_7    \\
e_8^2&=&e_3+e_5+e_7      \\
& \vdots &
\end{array}
}
\qquad
{
\begin{array}{lll}
\\
e_3^2&=&e_4+e_6\\
        e_5^2&=&e_2+e_4+e_6\\
           e_7^2&=&e_6+e_8\\
      e_9^2&=&e_4+e_6+e_8\\
& \vdots &
      \end{array}
      }
\end{displaymath}
Then $A={\rm lin}\{e_i^2\} $ for every $ i \in \N$,  but $A$ is not simple as it is not non-degenerate.
\end{example}

Another characterization of simplicity for finite dimensional evolution algebras is the following.

\begin{corollary}\label{CharSimpleFin}
If $A$ is a finite dimensional evolution algebra and $B$ a natural basis, then $A$ is simple if and only if $\vert M_B(A)\vert\neq 0$ and $\Lambda = D(i)$ for every $i \in \Lambda$.
\end{corollary}
\begin{proof}
Apply Theorem \ref{CharacSimple} (iv) taking into account that finite dimensionality of $A$ implies that $A= {\rm lin}\{e_i^2 \ \vert \ i \in \Lambda\}$ if and only if $\vert M_B(A)\vert\neq 0$.
\end{proof}

\begin{remark}
\rm
In terms of  graphs, the condition ``$\Lambda =  D(i)$" in Theorem \ref{CharacSimple} (iv) means that the graph associated to $A$ relative to a natural basis $B$ is cyclic, in the sense that given two vertices there is always a path from one to the other one.
\end{remark}

\medskip

The following remark shows how to get ideals in non-simple evolution algebras.

\begin{remark}
\rm
If $A$ is a non-degenerate evolution algebra having a natural basis $B=\{e_{i}\ \vert \ i\in
\Lambda \}$ such that every element $i\in \Lambda $ is a
descendent of every $j\in \Lambda$, then $A$ is not simple if and only if ${\rm lin}\{e_{i}^{2}\ \vert \ i\in \Lambda \}$ is a proper ideal of $A$.
\end{remark}


\section{Decomposition of an evolution algebra into a direct sum of evolution ideals}\label{DirectSum}


In this section we characterize the decomposition of any non-degenerate evolution algebra into direct summands as well as the non-degenerate irreducible evolution algebras in terms of the associated graph (relative to a natural basis).
When the graph associated to an evolution algebra and to a natural basis is non-connected then it gives a decomposition of the algebra into direct summands.
We define the optimal direct-sum decomposition of an evolution algebra and prove its existence and unicity when the algebra is non-degenerate.


When the algebra is finite dimensional we determine those elements in the associated graph relative to a natural basis (respectively in the algebra) which generate a decomposition into direct summands.

A decomposition of an evolution algebra can be seen, biologically, as a disjoint union of families of genotypes, each of these families reproduces only with the single individuals of the proper family.

\subsection{Reducible evolution algebras.}

In this subsection we are interested in the study of those evolution algebras which can be written as direct sums of (evolution) ideals.
\medskip
\begin{definition}\label{sumaDir}
\rm
Let $\{A_\gamma\}_{\gamma\in \Gamma}$ be a nonempty family of evolution $\mathbb K$-algebras. We define the \emph{direct sum} of these evolution algebras and denote it by $A:=\oplus_{\gamma\in \Gamma}A_\gamma$ with the following operations:
given $a=\dsum\limits_{\gamma\in \Gamma}a_\gamma,\ b=\dsum\limits_{\gamma\in \Gamma}b_\gamma\in A$ and $\alpha \in \mathbb K$ (note that $a_\gamma$ and $b_\gamma$ are zero for almost every $\gamma \in \Gamma$), define

\begin{equation*}
a+b:=\dsum\limits_{\gamma\in \Gamma}\left( a_\gamma+b_\gamma\right),\qquad
\alpha a:=\dsum\limits_{\gamma\in \Gamma}\left(\alpha a_\gamma\right), \qquad
ab:=\dsum\limits_{\gamma\in \Gamma}\left(a_\gamma b_\gamma\right).
\end{equation*}
Note that $A$ is an evolution algebra as, if $B_\gamma$ is a natural basis of $A_\gamma$ for every $\gamma\in \Gamma$, then  $B:= \cup_{\gamma\in \Gamma}B_\gamma$ is a natural basis of $A$. Here, by abuse of notation, we understand $A_\gamma \subseteq A$ so that every $A_\gamma$ can be regarded as an (evolution) ideal of $A$. Moreover, for $\gamma \neq \mu$, the ideals $A_\gamma$ and $A_\mu$ are orthogonal, in the sense that $A_\gamma A_\mu = 0.$
\end{definition}

\medskip

\begin{lemma}\label{reduci}
Let $A$ be an evolution algebra. The following assertions are
equivalent:
\begin{enumerate}[\rm (i)]
\item There exists a family of evolution subalgebras $\{A_{\gamma}\}_{\gamma\in\Gamma}$ such that
$A=\oplus_{\gamma\in\Gamma}A_\gamma.$

\item  There exists a family of evolution ideals $\{I_{\gamma}\}_{\gamma\in\Gamma}$ such that
$A=\oplus_{\gamma\in\Gamma}I_\gamma.$

\item  There exists a family of  ideals $\{I_{\gamma}\}_{\gamma\in\Gamma}$ such that
$A=\oplus_{\gamma\in\Gamma}I_\gamma.$
\end{enumerate}
\end{lemma}

\begin{proof}
(i) $\Rightarrow$ (ii). By the definition of direct sum of evolution algebras (see Definition \ref{sumaDir}), every $A_\gamma$ is, in fact, an evolution ideal.

(ii) $\Rightarrow$ (iii) is a tautology.

(iii) $\Rightarrow$ (i). Suppose  $A=\oplus_{\gamma\in\Gamma}I_\gamma$, where each $I_\gamma$ is an ideal of $A$. For  $\mu\in \Gamma$ we have:
$$I_\mu \cong A/\left(\oplus_{\gamma\in\Gamma\setminus\{\mu\}}I_\gamma\right).$$

\noindent
By Lemma \ref{quotient} we obtain that $A/\left(\oplus_{\gamma\in\Gamma\setminus\{\mu\}}I_\gamma\right)$ is an evolution algebra, and hence $I_\mu$ is an evolution algebra  by Corollary \ref{coroquotient}.
\end{proof}

\begin{definition}
\rm
 A \emph{reducible evolution algebra} is an evolution algebra $A$
which can be decomposed as the direct sum (in the sense of Definition \ref{sumaDir}) of two non-zero evolution algebras, equivalently, of two non-zero evolution ideals, equivalently, of two non-zero ideals, as shown in Lemma \ref{reduci}. An evolution algebra which is not reducible will be called \emph{irreducible}.
\end{definition}

Reducibility of an evolution algebra is related to the connection of the underlying graphs, as we show next.
For the description of the (existent) connected components of a graph see, for example, \cite[Definitions 1.2.13]{AASbook}.

\medskip

\begin{proposition}\label{CompCon}
Let $A$ be a non-zero evolution algebra and $E$ its associated graph relative to a natural basis $B=\{e_i \ \vert \ i\in \Lambda\}$.
\begin{enumerate}[\rm (i)]

\item
Assume $E=E_1\sqcup E_2$, where $E_1$ and $E_2$ are nonempty subgraphs of $E$. Write $E_k^0=\{v_i \ \vert \ i \in \Lambda_k\}$, for $k=1, 2$, where $\Lambda_k \subseteq \Lambda$ and $\Lambda = \Lambda_1 \sqcup \Lambda_2$. Then there exist non-zero evolution ideals $I_1, I_2$ of $A$ such that $A= I_1 \oplus I_2$ and  $E_1,$ $E_2$ are the graphs associated  to the evolution algebras $I_1$ and $I_2$, respectively, relative to their natural basis  $B_k=\{e_i \ \vert \ i \in \Lambda_k\}$ (for $k=1, 2$). Moreover,  $B=B_1 \sqcup B_2$.

\item Let $E=\sqcup_{\gamma\in \Gamma}E_\gamma$ be the decomposition of $E$ into its connected components. For every $\gamma \in \Gamma$, write $E_\gamma^0= \{v_i \ \vert \ i \in \Lambda_\gamma\}$, where $\Lambda_\gamma \subseteq \Lambda$ and $\Lambda =\sqcup_{\gamma\in \Gamma} \Lambda_\gamma$. Then there exist  $\{I_\gamma\}_{\gamma\in \Gamma}$, evolution ideals of $A$, such that $A=\oplus_{\gamma\in \Gamma}I_\gamma$ and $E_\gamma$ is  the associated graph to the evolution algebra $I_\gamma$ relative to the natural basis $B_\gamma$ described below.
Moreover:
 \begin{enumerate}[\rm (a)]
\item $B=\sqcup_{\gamma\in \Gamma} B_\gamma$, where $B_\gamma=\{e_i \ \vert \ i \in \Lambda_\gamma\}$ is a natural basis of $I_\gamma$, for every $\gamma \in \Gamma$.

\item $I_\gamma$ is a simple evolution algebra if and only if $I_\gamma = {\rm lin}\{e_i^2 \  \vert \ i \in \Lambda_\gamma\}$ and $D(i) = \Lambda_\gamma$ for every $i\in \Lambda_\gamma$.

\item $A$ is non-degenerate if and only if every $I_\gamma$ is a non-degenerate evolution algebra.
\end{enumerate}
\end{enumerate}
\end{proposition}

\begin{proof} (i).
Let $B=\{e_i \ \vert \ i \in \Lambda\}$. The decomposition $E=E_1 \sqcup E_2$ of $E$ into two non empty components provides a decomposition  of $A$ into ideals as follows. Denote by $v_i$, with $i\in \Lambda$, the vertices of $E$. Write $E^0=E_1^0\sqcup E_2^0$, and let $\Lambda_k\subseteq \Lambda$ be such that $\Lambda_k=\{i \in \Lambda \ \vert \ v_i\in E_k^0\}$, for $k=1, 2$. Define $I_k= {\rm lin}\{e_i \ \vert \ i\in \Lambda_k\}$. Then $A=I_1\oplus I_2$. The moreover part follows easily.

(ii). The first part can be proved as (i). Item (a) follows immediately. As for
(b), apply Theorem \ref{CharacSimple} (\ref{algo4}). To prove (c) use Proposition \ref{degen} and Corollary \ref{MVM}.
\end{proof}

\begin{remark}\label{ay}
\rm Once we have defined what an optimal direct-sum decomposition is (see Definition \ref{opti}) we can say that if $A$ is non-degenerate then $A=\oplus_{\gamma\in \Gamma}I_\gamma$ in Proposition \ref{CompCon} (ii) is the optimal direct sum decomposition of $A$, as will follow from Theorem \ref{uniopti}.
\end{remark}

In the next result we characterize when a non-degenerate evolution
algebra $A$ is reducible, giving an answer to one of our main questions in this work.

\begin{theorem}\label{caracteriz}
Let $A$ be a non-degenerate evolution algebra with a natural basis $B=\{e_i \ \vert \ i\in \Lambda\}$ and assume
that $A= \oplus_{\gamma\in \Gamma}I_\gamma$, where each $I_\gamma$ is an ideal of $A$. Then:

\begin{enumerate}[\rm (i)]
\item For every $e_i\in B$ there exists a unique $\mu\in \Gamma$ such that $e_i\in I_\mu$. Moreover, $e_i\in I_\mu$ if and only if $e_i^2\in I_\mu$.
\item There exists a disjoint decomposition of $\Lambda$, say $\Lambda=\sqcup_{\gamma\in\Gamma} \Lambda_\gamma$, such that $$I_\gamma = {\rm lin}\{e_i\ \vert \ i\in \Lambda_\gamma\}.$$
\end{enumerate}
\end{theorem}

\begin{proof} We show both statements at the same time.
Let $\pi _{i}$ be the linear projection of $A$ over $\mathbb{K}e_{i}.$
We show first that  $\pi _{i}(I_\gamma)\neq 0$ implies $e_{i}^{2}\in I_\gamma,$ and hence $\pi
_{i}(I_\mu)=0$ for every $\mu\in \Gamma\setminus\{\gamma\}$. Indeed, if $\pi _{i}(I_\gamma)\neq 0$, then
there exists $y \in I_\gamma$ such that $\pi_i(y) = \alpha e_i\neq 0$ for some $\alpha\in \mathbb K$.
Multiplying by $e_i$ we get $e_iy= e_i\pi_i(y) = \alpha e_i^2\in I_\gamma$ and, therefore, $e_i^2\in I_\gamma$.
If  $\pi _{i}(I_\mu)\neq 0$ for some $\mu \in \Gamma$, reasoning as before, we get $e_i^2\in I_\mu$, and so $e_i^2 \in I_\gamma \cap I_\mu = 0$, a contradiction because we are assuming that $A$ is non-degenerate.

Define $\Lambda_\gamma : =\{ i \in \Lambda \ \vert \ \pi_i(I_\gamma)\neq 0\}$. It is easy to see that
$\cup_{\gamma\in\Gamma}\Lambda_\gamma = \Lambda$. Moreover, the first paragraph of the proof shows that this is a disjoint union, as claimed in (ii).

Now, it is easy to see that for every $\gamma \in \Gamma$ we have that $I_\gamma\subseteq {\rm lin}\{e_{i}\ \vert \ i\in \Lambda _{\gamma}\}$.
To show that ${\rm lin}\{e_{i}\ \vert \ i\in \Lambda _{\gamma}\}\subseteq I_\gamma $, consider $e_j\in B$, with $j\in \Lambda_\gamma$ and denote $J = \oplus_{\mu\in\Gamma\setminus\{\gamma\}}I_\mu$.
Because $A= I_\gamma \oplus J$ we may write  $e_j=u+v$, with $u\in I_\gamma$ and $v\in J$. Then, $v= e_j-u\in  {\rm lin}\{e_{i}\ \vert \ i\in \Lambda _{\gamma}\}$ because $e_j$ and $u$ are in ${\rm lin}\{e_{i}\ \vert \ i\in \Lambda _{\gamma}\}$. Since $v\in J \subseteq {\rm lin}\{e_{i}\ \vert \ i\in \cup_{\mu\in \Gamma \setminus\{\gamma\}}\Lambda _{\mu}\}$ we deduce that $v$ must be zero.
\end{proof}

\begin{remark}
\rm
Theorem \ref{caracteriz} gives another proof, for non-degenerate evolution algebras, of the fact that if an evolution algebra is a direct sum of ideals, then such ideals are evolution algebras (and, consequently, evolution ideals). This is the assertion (ii) $\Leftrightarrow$ (iii) established in Lemma \ref{reduci}.
\end{remark}

Another application of  Theorem \ref{caracteriz} allows us to recognize easily when a non-degenerate finite dimensional evolution algebra $A$ is reducible: if $B=\{e_{i}\ \vert \ i=1,..., n\}$ is a natural basis
of $A$ then $A$ is the direct sum of two (evolution) ideals if and only if there is a permutation $\sigma\in S_n$ such that, if $B':= \{e_{\sigma(i)}\ \vert \ i=1,..., n\}$, then the corresponding structure matrix is

$$
M_{B'} = \left(
\begin{array}{cc}
W_{m\times m} & 0_{(n-m)\times (n-m)} \\
0_{(n-m)\times m} & Y_{(n-m)\times (n-m)}
\end{array}
\right),
$$

\medskip

\noindent
for some $m\in \N$, $m < n$ and  some matrices  $W_{m\times m} $ and $Y_{(n-m)\times (n-m)}$ with entries in $\mathbb K$.
In this case $A=I\oplus J$, where $I={\rm lin}\{e_{\sigma(1)},...,e_{\sigma(m)}\}$ and $
J={\rm lin}\{e_{\sigma(m+1)},...,e_{\sigma(n)}\}$. The basis $B'$ is what we will called a \emph{reordering} of $B$.
\medskip

\begin{corollary}\label{conexo}
Let $A$ be a non-degenerate evolution algebra,  $B=\{e_i\ \vert \ i\in \Lambda\}$ a natural basis, and let $E$ be its associated graph. Then $A$ is irreducible if and only if $E$ is a connected graph.
\end{corollary}
\begin{proof}
Suppose first that $E$ is connected. To show that $A$ is irreducible suppose, on the contrary, that there exist $I$ and $J$,  non-zero ideals of $A$, such that $A=I\oplus J$. By Theorem \ref{caracteriz} there exists a decomposition $\Lambda= \Lambda_I \sqcup \Lambda_J$ such that
$I= {\rm lin}\{e_i\ \vert \ i \in \Lambda_I\}$ and $J= {\rm lin}\{e_i\ \vert \ i \in \Lambda_J\}$. Then $E= E_I \sqcup E_J$, a contradiction since we are assuming that $E$ is connected.

The converse follows easily: by Proposition \ref{CompCon} (i),
a decomposition $E=E_1 \sqcup E_2$ into two non empty components provides a decomposition  $A=I_1\oplus I_2$, for $I_1$ and $I_2$ non-zero ideals of $A$, contradicting that $A$ is irreducible.
\end{proof}

In \cite[Proposition 2.8]{EL} the authors show the result above for finite-dimensional evolution algebras using a different approach.

\medskip

The hypothesis of non-degeneracy cannot be eliminated in Corollary \ref{conexo}.

\begin{example}\rm
Consider the evolution algebra given in Example \ref{cambio}, which is not non-degenerate. Then the graph $E$, associated to the basis $B$ is connected while the graph $F$, associated to the basis $B'$ is not.
\end{example}


\subsection{The optimal direct-sum decomposition of an evolution algebra}\label{Descomposicion}

\medskip

The aim of this subsection is to obtain a decomposition of an evolution algebra in terms of irreducible evolution ideals.

\begin{definition}\label{opti}
\rm
Let $A$ be a non-zero evolution algebra and assume that $A=\oplus_{\gamma\in\Gamma}I_\gamma$ is a direct sum of non-zero ideals. If every $I_\gamma$ is an irreducible evolution algebra, then we say that $A=\oplus_{\gamma\in\Gamma}I_\gamma$  is an
 \emph{optimal direct-sum decomposition} of $A$.
\end{definition}

We  show that the optimal direct sum decomposition of an evolution algebra $A$  with a natural basis $B=\{e_i \ \vert i \in \Lambda\}$ does exist and it is unique whenever the algebra is non-degenerate. Moreover, for finite dimensional evolution algebras (degenerated or not), we will describe how to get an optimal decomposition of $\Lambda$ through the fragmentation process. This will be done in Subsection \ref{fragme}

\begin{theorem}\label{uniopti}
Let $A$ be a non-degenerate evolution algebra. Then $A$ admits an optimal direct-sum decomposition. Moreover, it is unique.
\end{theorem}

\begin{proof}
We start by showing the existence. Let $E$ be the graph associated to $A$ relative to a natural basis $B$ and decompose it in its connected components, say $E=\sqcup_{\gamma\in \Gamma} E_\gamma$. By Proposition \ref{CompCon} (ii) we have $A=\oplus_{\gamma\in \Gamma} I_\gamma$, where every $I_\gamma$ is an ideal of $A$. Note that, by construction (see the proof of Proposition \ref{CompCon}), every $I_\gamma$ has a natural basis, say $B_\gamma$, consisting of elements of the basis $B=\{e_i \ \vert \ i \in \Lambda\}$ of $A$. Because $A$ is non-degenerate, $e_i^2\neq 0$, by Corollary \ref{MVM} and Proposition \ref{degen}. Using again these results we have that  every $I_\gamma$ is a non-degenerate evolution algebra.
Since $E_\gamma$ is the graph associated to $I_\gamma$ relative to the basis $B_\gamma$ and $E_\gamma$ is connected, by Corollary \ref{conexo} every $I_\gamma$ is an irreducible evolution algebra.

Now we prove the uniqueness.
Fix a natural basis $B= \{e_i\ \vert \ i\in \Lambda\}$.
Suppose that there are two optimal direct-sum decompositions of $A$, say $A=\oplus_{\gamma\in \Gamma}I_\gamma$ and $A=\oplus_{\omega\in \Omega}J_\omega$. By  Theorem \ref{caracteriz} there exist two decompositions $\Lambda = \sqcup_{\gamma\in \Gamma}\Lambda_\gamma$ and $\Lambda = \sqcup_{\omega\in \Omega}\Lambda_\omega$ such that

$$I_\gamma = {\rm lin}\{e_i\ \vert \ i\in \Lambda_\gamma\} \quad \text{and}\quad
J_\omega = {\rm lin}\{e_i\ \vert \ i\in \Lambda_\omega\}.$$

\noindent
Take $i\in \Lambda_\gamma$ for an arbitrary $\gamma \in \Gamma$. Then there is an  $\omega\in\Omega$ such that $e_i\in J_\omega$. This means $I_\gamma\cap J_\omega\neq 0$. Decompose

$$I_\gamma = (I_\gamma \cap J_\omega)\oplus \left(I_\gamma\cap \left(\oplus_{\omega\neq \omega'\in \Omega}J_{\omega'}\right)\right).$$
Since $I_\gamma$ is irreducible and $I_\gamma\cap J_\omega\neq 0$, necessarily
$\left(I_\gamma\cap \left(\oplus_{\omega\neq \omega'\in \Omega}J_{\omega'}\right)\right)=0$. Therefore $I_\gamma = I_\gamma \cap J_\omega$ and so $I_\gamma \subseteq J_\omega$. Changing the roles of $I_\gamma$ and $J_\omega$ we get $J_\omega \subseteq I_\gamma$, implying
$I_\gamma=J_\omega$ and, consequently, that each decomposition is nothing but a reordering of the other one.
\end{proof}
\medskip

The hypothesis of non-degeneracy cannot be eliminated in order to assure the unicity of the optimal direct sum decomposition in Theorem \ref{uniopti}, as the following example shows. It is also an example which illustrates that  in Theorem \ref{caracteriz} non-degeneracy is also required.

\begin{example}\label{new}
\rm
Let $A$ be the evolution $\mathbb K$-algebra with natural basis
$B=\{e_{1},e_{2},e_{3},e_{4}, e_{5}\}$ and multiplication given by:
$e_{1}^{2}=e_{2}^{2}=e_{1},$ $e_{3}^{2}=e_{3}+e_{5}$ and
$e_{4}^{2}=e_{5}^{2}=0.$
 Then
 $A=I_{1}\oplus I_{2}\oplus I_{3}\oplus I_{4}$,  where
$I_{1}:={\rm lin}\{e_1, e_{2}+ e_{4}\}$,
$I_{2}:={\rm lin}\{e_{3}+e_{5}\}$,
$I_{3}:={\rm lin}\{e_{4}\}$ and
$I_{4}:={\rm lin}\{e_{5}\}$ are irreducible ideals, as we are going to show.

The ideals $I_{2},$ $I_{3}$ and $I_{4}$ are irreducible because their
dimension is one.  Now we prove that $I_1$ is also irreducible.
Assume, on the contrary, $I_{1}=J_{1}\oplus J_{2}$, with $J_1$ and $J_2$ non-zero ideals.
Then $\dim J_{1}=\dim J_{2}=1,$ so that $J_{1}=\mathbb Ku_{1}$ and $J_{2}=\mathbb K u_{2}$ for some $u_{1}=\alpha _{1}e_{1}+ \beta_{1}(e_{2}+e_{4})$ and
$u_{2}=\alpha _{2}e_{1}+\beta_{2}(e_{2}+e_{4})$, where $\alpha_1, \alpha_2, \beta_1, \beta_2 \in \mathbb K$.
Then, $u_1u_2 = 0$ implies $(\beta_1\beta_2+\alpha_1\alpha_2)e_1 =0$. On the other hand,
$u_1e_1=\alpha_1e_1\in J_1$ and $u_2e_1=\alpha_2e_1\in J_2$. Since $J_1\cap J_2=0$, then
$\alpha_1=0$ or $\alpha_2=0$. Assume, for example, $\alpha_1=0$. Then $J_1= \mathbb K (e_2+e_4)$, but
this is not an ideal as $(e_2+e_4)^2 = e_1$. The case $\alpha_2=0$ is similar.

 Now we give another decomposition of $A$ into irreducible ideals. Consider
 $A=J\oplus I_{2}\oplus I_{3}\oplus I_{4}$, where
 $J:={\rm lin}\{e_{1}, e_{2}\}$.  We claim that  $J$ is an irreducible ideal of $A$. Indeed, if
 $J=M_{1}\oplus M_{2}$, for $M_1$ and $M_2$ non-zero ideals, then $M_{1}=\mathbb Ku_{1}$
  and $M_{2}=\mathbb Ku_{2}$ for some
  $u_{1}=\alpha _{1}e_{1}+\beta _{1}e_{2}$ and $u_{2}=\alpha _{2}e_{1}+\beta _{2}e_{2}$,
 where $\alpha_1, \alpha_2, \beta_1, \beta_2 \in \mathbb K$.
  Then, $u_1e_1=\alpha_1e_1$ and $u_2e_1=\alpha_2e_1$. Since $M_1\cap M_2=0$, then
  $\alpha_1=0$ or $\alpha_2=0$. Assume $\alpha_1=0$. This implies $u_1=\beta_1e_2$ and $M_1= \mathbb K e_2$, but  this is not an ideal because $e_2^2=e_1$. The case $\alpha_2=0$ is similar.

  Note that we have two different decompositions of $A$ as a direct sum of irreducible ideals.
  \end{example}

As we have seen (Remark \ref{noSimple}), non-degenerate evolution algebras are not necessarily simple. On the other hand, concerning reducibility, the next example shows that there exists irreducible
evolution algebras which are not simple (while, obviously, simple evolution algebras are irreducible).

\begin{example}
\label{redus-simple}\emph{Let }$A$\emph{\ be an evolution algebra with
a natural basis }$B=\{e_{1},e_{2}\}$\emph{\ such that }$%
e_{1}^{2}=e_{2}^{2}=e_{2}.$\emph{\ Then }$\mathbb{K}e_{2}$\emph{\ is a
proper ideal of }$A.$\emph{\ However,} $A$\emph{\ is irreducible because
if }$A=I\oplus J$, \emph{ for some ideals } $I$ \emph{and } $J$\emph{ of $A$. Then, by
Theorem \ref{caracteriz}, we have either }$e_{1}\in I,$\emph{\ in which case
}$A=I,$\emph{\ or }$e_{1}\in J,$\emph{\ in which case, }$A=J.$\emph{\ In any case
$I$ or $J$ is zero.}
\end{example}

\medskip
The next definition will be helpful to understand the inner structure of an
evolution algebra.

\begin{definition}\label{ciclico}
\rm
Let $B=\{e_{i}\ \vert \ i\in \Lambda \}$ be a natural basis of an
evolution algebra $A.$ We say that $i_{0}\in \Lambda $ is
\textit{cyclic} if $i_{0}\in D(i_{0}).$ This means that $i_{0}$
is descendent (and hence ascendent) of itself$.$

In particular, if $D(i_{0})=\{i_{0}\}$ (in which case
$e_{i_{0}}^{2}=\omega _{i_{0}i_{0}}e_{i_{0}}$ for some $\omega_{i_{0}i_{0}}\in \mathbb K\setminus \{0\}$), then we say that the
cyclic index $i_{0}$ is a \textit{loop.}

If $i_{0}\in \Lambda $ is cyclic, then the
\textit{cycle associated to $i_{0}$} is defined as the set:
$$
C(i_{0}) = \{j\in \Lambda \ \vert \ j\in D(i_{0})\ \text{\rm and }i_{0}\in D(j)\}.
$$

Note that if $i_{0}$ is cyclic then
$C(i_{0})$ is non-empty because it contains $i_{0}$ in particular. Moreover,
$i_{0}$ is a loop if and only if $C(i_{0})=\{i_{0}\}.$

We say that a subset $C \subseteq \Lambda $ is a
\textit{cycle} if $C=C(i_{0}),$ for some cyclic-index $i_{0}\in \Lambda.$
\end{definition}

\begin{remark}
\rm

By identifying an index $i$ with the genotype $e_i$, biologically, an index is cyclic if it is a descendent of its descendents. The cycle associated to an index $i$ is the set of all its descendents $j$ such that $i$ is a descendent of $j$.
\end{remark}

\medskip

In the same context as in Definition \ref{ciclico}, consider $i_0\in \Lambda$, and let
$\omega _{i_{0}i_{0}}$ be the corresponding element in the structure matrix for the evolution
algebra $A$. If $\omega _{i_{0}i_{0}}\neq 0$ then we have that
$i_{0} $ is cyclic, independently of the value of the other elements in the structure
matrix. If $\omega _{i_{0}i_{0}}=0,$ then $i_{0}$ is cyclic if and only
if it is a descendent of some of its own descendents.

On the other hand, if $B=\{e_i\ \vert \ i\in \Lambda \}$ is a natural basis of
$A$,  and if $i_{1},i_{2}\in \Lambda $ are cyclic, then
we have either $C(i_{1})=C(i_{2})$ or $C(i_{1})\cap C(i_{2})=\emptyset .$

These facts can be understood more easily looking at the corresponding
graphical concepts, as we will do below.
\medskip

Now we classify the cycles into two types, depending on if they
have or not ascendents outside the cycle.

\begin{definition}\label{cicle}
\rm
Let $B=\{e_{i}\ \vert \ i\in \Lambda \}$ be a natural basis
of an evolution algebra $A,$ and let $i_{0}\in \Lambda $ be a
cyclic index. We say that $i_0$ is a \emph{principal cyclic index} if the set of
ascendents of $i_{0}$ is contained in $C(i_{0})$, the cycle associated to $i_0$. Thus, $i_{0}\in \Lambda $ is a principal cyclic-index if $i_{0}\in D(i_{0})$ and $j\in D(i_{0})$ for every $j\in \Lambda $ with $i_{0}\in D(j)$.

We say that a subset $C$ of $\Lambda $ is a \emph{principal cycle} if $C=C(i_{0})$, for some principal cyclic index
$i_{0}\in \Lambda$.
\end{definition}
\medskip

It is clear that if $i_{0}\in \Lambda $ is a principal
cyclic index then every $j\in C(i_{0})$ is also a principal
cyclic index. Moreover, if $i_{0}\in \Lambda $ is a cyclic index,
then $C(i_{0})$ is not principal if and only if there exists
$j\in \Lambda \setminus C(i_0)$ such that $i_{0}\in D(j)$.

On the other hand, a non-empty subset $C\subseteq \Lambda $
is a principal cycle if and only if it satisfies the following properties:

\begin{enumerate} [\rm (i)]
\item For every $i,j\in C$ we have that $i\in D(j)$ and $j\in D(i).$
\item If $D(k)\cap C\neq \emptyset $ then $k\in C$.
\end{enumerate}

Note that if $i_{0}$ is a loop, then $\{i_{0}\}$ is
a principal cycle if and only if $i_0$ has no other ascendents than $i_0$.
Moreover, if $C$ is a principal cycle, then
$C=C(i)=C(j)$ for every $i, j\in C$ and, hence,
$D(i)=D(j)$ for very $i,j\in C.$

\medskip

Now we will distinguish between cycles that have proper descendents from
those that do not have them.

\begin{definition}\label{principal}
\rm
Let $B=\{e_{i}\ \vert \ i\in \Lambda \}$ be a natural
basis of an evolution algebra $A$, and let $S$ be a subset
of $\Lambda$. We define the \textit{index-set derived from}
$S$ as the set given by
$$
\Lambda (S):=S\cup _{i\in S}D(i).
$$
\end{definition}

For instance, if $i\in \Lambda$, then the index set derived from $\{i\}$ is
$
\Lambda (\{i\}):=\{i\}\cup D(i),
$
where $D(i)$ is the set of descendents of $i.$

The index set derived from a principal cycle is obtained next.

\begin{remark}
\rm
Let $C$ be a principal cycle. Then, $C=C(i)=C(j)$
and $D(i)=D(j)$, for every $i,j\in C.$
Moreover, $C(i)\subseteq D(i),$ for every $i\in C$
(the inclusion may or not be strict). Thus, $C\subseteq \Lambda (C)=D(i)$ for every $i\in C.$
\end{remark}

The Definition \ref{ciclico} in terms of graphs gives rise to the following definition.

\begin{definition}\label{ciclos}
\rm
Let $E$ be a  graph with vertices $\{v_i \ \vert \ i\in \Lambda\}$ satisfying Condition (Sing).
An index $j\in \Lambda$ is  said to be  \emph{cyclic} if $j\in D^m(j)$ for some $m\in \N$ (see Definitions \ref{descendientes}). Equivalently, if the graph $E$ has a cycle $c$ such that $v_j\in c^0$.

If $j$ is a cyclic index, we define
$$C(j) : = \{v_k \in E^0\ \vert \  k \in D(j) \text{ and } j \in D(k)\},$$
that is, $C(j)$ are those vertices connected to $v_j$ such that $v_j$ is also connected to them.

A cyclic index $j$ is called \emph{principal}  (see Definition \ref{cicle}) if
$$\{v_k\in E^0 \ \vert \ j\in D(k)\} \subseteq C(j).$$
A cyclic index $j$ is \emph{principal} if and only if it belongs to a cycle without entries or such that every entry comes from a path starting at the cycle.
By extension, we will also say that $C(j)$ is a \emph{principal cycle}.
\end{definition}
\medskip

\begin{examples}\label{ej3.13}
\rm
Consider the following graphs.

$$ E:
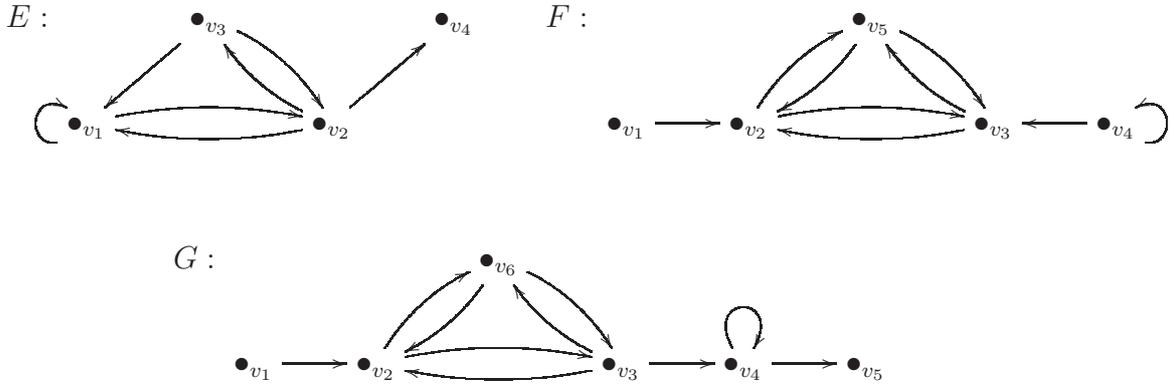
\xymatrix{   & \bullet_{v_3} \ar@/^.0pc/[dl] \ar@/^.5pc/[dr] &  & \bullet_{v_4}\cr
            \bullet_{v_1} \ar@(dl,ul) \ar@/^.5pc/[rr]  & &\bullet_{v_2}\ar@/^.5pc/[ll] \ar@/^.5pc/[ul] \ar@/^.0pc/[ur]
            }
             \quad \quad
F:
\xymatrix{
 & & \bullet_{v_5} \ar@/^.5pc/[dl] \ar@/^.5pc/[dr] & & \cr
  \bullet_{v_1}  \ar[r]   &      \bullet_{v_2}  \ar@/^.5pc/[ur] \ar@/^.5pc/[rr]  & &\bullet_{v_3} \ar@/^.5pc/[ll] \ar@/^.5pc/[ul] &   \bullet_{v_4}\ar[l] \ar@(dr,ur)
}
$$
\vspace{.10truecm}

$$
G: \xymatrix{
 & & \bullet_{v_6} \ar@/^.5pc/[dl] \ar@/^.5pc/[dr] & & & \cr
  \bullet_{v_1}  \ar[r]   &      \bullet_{v_2}  \ar@/^.5pc/[ur] \ar@/^.5pc/[rr]  & &\bullet_{v_3} \ar@/^.5pc/[ll] \ar@/^.5pc/[ul] \ar[r] &   \bullet_{v_4}  \uloopr{}  \ar[r] &  \bullet_{v_5}&
}
$$
\vspace{.10truecm}

Concerning $E$, the indices $1, 2$ and $3$ are cyclic and $C(1)=C(2)=C(3) = \{v_1, v_2, v_3\}$. Also, $1, 2$ and $3$ are principal indices.

For the graph $F$, the cyclic indices are $2, 3, 4$ and $5$. Moreover, $C(2)=C(3)=C(5) = \{v_2, v_3, v_5\}$ and $C(4)=\{v_4\}$. The only index which is principal is  $4$.

As to the graph $G$, its cyclic indices are $2, 3, 4$ and $6$. None of them is principal.
\end{examples}

\begin{definition} \label{chain-start}
\rm
Let $B=\{e_{i}\ \vert \ i \in \Lambda \}$ be a
natural basis of an evolution algebra $A.$
 We say that $i_{0}\in \Lambda $ is a \textit{chain-start index} if $i_{0}$ has no
ascendents, i.e., $i_{0}\notin D(j),$ for every $j\in \Lambda$. Equivalently,
$i_{0}$ is a chain-start index if and only if all
the elements of the $i_{0}$-th row of the structure matrix
$M_B(A)$ are zero.
\end{definition}

\begin{remark}
\rm
In terms of graphs, an index $i_0$ is a chain-start index if and only if the vertex $v_{i_0}$ is a source.
In the graphs of Examples \ref{ej3.13}, the only chain-start index is the vertex $v_1$ in $F$ and the vertex $v_1$ in $G$.
\end{remark}

In the case of finite-dimensional evolution algebras, if the determinant of the structure matrix $M_B(A)$ is non-zero then $\Lambda $ has no chain-start indices. The (graphical) reason is that $\vert M_B(A)\vert \neq 0$ implies that $M_B(A)$ has no zero rows, hence
 the associated graph relative to $B$ has no sources.

 The lack of chain-start indices is a necessary condition for $A$ to be simple. The graphical reason is that if a vertex $v_i$ is a source, then the ideal generated by $e_i^2$ does not contains $e_i$, hence $\langle e_i^2\rangle$  is a nonzero proper ideal.

\medskip

\subsection{The fragmentation process\label{fragme}}

In this subsection we will consider only finite dimensional evolution algebras, degenerated or not. We give a process that allow to decompose an evolution algebra into direct sums of evolution algebras (the optimal decomposition when the algebra is non-degenerate).

\begin{definition}
\rm
Let $A$ be a finite dimensional evolution algebra, and fix a natural basis $B=\{e_{i}\ \vert \ i \in \Lambda \}.$
Consider the set $\{C_{1}, \dots, C_{k}\}$  of the
 principal cycles of $\Lambda$ and the set $\{i_{1}, \dots, i_{m}\}$
of all chain-start indices of $\Lambda.$

Given any $i\in \Lambda $ which is not a chain-start index, there
exists $j\in \Lambda $ such that $i\in D(j),$ and either $j$ is a
chain-start index or $j$ belong to a principal cycle (because $\Lambda $ is
finite).
Therefore, according to Definition \ref{principal},
$$
\Lambda =\Lambda (C_{1})\cup \dots\cup \Lambda (C_{k})\cup \Lambda (i_{1})\cup \dots \cup
\Lambda (i_{m}).  \label{frag}
$$
This decomposition will be called the \emph{canonical decomposition of} $\Lambda$
\emph{associated to} $B$.
\end{definition}

Note that the sets in the canonical decomposition are not necessarily disjoint. This is the case,
for example, when two different principal cycles, or two chain-start indices, have common descendents.

\begin{definition}
\rm
Let $\Lambda $ be a finite set and let
$\Upsilon_{1},\dots ,\Upsilon _{n}$ be non-empty subsets of $\Lambda $
such that $\Lambda =\cup _{i=1}^{n}\Upsilon _i.$
We say that $\Lambda =\cup _{i=1}^{n}\Upsilon _i$ is a
\emph{fragmentable union} if there exists disjoint non-empty subsets
$\Lambda_{1}, \Lambda _{2}$ of $\Lambda $ satisfying
$$
\Lambda =\cup _{i=1}^{n}\Upsilon _i=\Lambda _{1}\cup \Lambda
_{2},
$$
and such that for every $i=1,\dots, n$, either
$\Upsilon _{i} \subseteq \Lambda _{1}$ or
$\Upsilon _{i} \subseteq \Lambda _{2}.$
\end{definition}

For instance, if the sets $\Upsilon _{i}$ are disjoint then $\Lambda =\cup
_{i=1}^{n}\Upsilon _{i}$ is fragmentable. Note that a fragmentable union
may admit different fragmentations.

On the other hand, if $\Upsilon _{i}\cap \Upsilon _{j}\neq \emptyset $ for every $i\neq j$, then the union $\Lambda
=\cup _{i=1}^{n}\Upsilon _{i}$ is not fragmentable.

\begin{definitions}\label{opti-fragme}
\rm
Let $\Lambda $ be a finite set and let
$\Upsilon _{1}, \dots,\Upsilon _{n}$ be non-empty subsets of $\Lambda $
such that $\Lambda =\cup _{i=1}^{n}\Upsilon _i$
is a fragmentable union. A  \emph{fragmentation of }
$\Lambda =\cup _{i=1}^{n}\Upsilon _i$ is a union
$\Lambda=\cup _{i=1}^{k}\Lambda _i$ such that:
\begin{enumerate}[\rm (i)]
\item If $i\in \{1, \dots,k\}$ then $\Lambda _{i}=\cup _{j\in S_{i}}\Upsilon _{j}$ for $S_i$ a non-empty subset of $\{1,\dots,n\}$.

\item $\Lambda _{i}\cap \Lambda _{j}=\emptyset ,$ for every
$i,j\in \{1,\dots,k\}$, with $i\neq j$.
\end{enumerate}
Note that  conditions (i) and (ii) imply that for every
$j\in \{1,\dots,n\}$ there exists a unique $i\in \{1,\dots,k\}$
such that $\Upsilon _{j}\subseteq \Lambda _{i}$.

An \emph{optimal fragmentation} of a fragmentable union
$\Lambda=\cup _{i=1}^{n}\Upsilon _i$ is a fragmentation
$\Lambda =\cup _{i=1}^{k}\Lambda _i$ such that for every
$i\in \{1,\dots,k\}$ the index set $\Lambda _{i}=\cup _{j\in S_{i}}\Upsilon _{j}$ is not
fragmentable.
\end{definitions}

\medskip
In what follows we build the optimal fragmentation for any $\Lambda=\cup_{i=1}^n\Upsilon_i$.
\medskip

Let $\Lambda $ be a finite set and consider $\Upsilon _{1},\dots,\Upsilon _{n}$,
non-empty subsets of $\Lambda$ such that $\Lambda =\cup _{i=1}^{n}\Upsilon _i$
is a fragmentable union. To obtain an optimal
fragmentation of this union we
define the following equivalence relation in the set
$\{\Upsilon_{1},\dots,\Upsilon _{n}\}.$

We say that $\Upsilon _{i}\sim \Upsilon _{j}$ if there exist
$m_{1},\dots,m_{k}\in \{1,\dots,n\}$ such that
$$
\Upsilon _{i}\cap \Upsilon _{m_{1}}\neq \emptyset ,\;\Upsilon _{m_{1}}\cap
\Upsilon _{m_{2}}\neq \emptyset ,\dots,\;\Upsilon _{m_{k-1}}\cap \Upsilon
_{m_{k}}\neq \emptyset ,\ \Upsilon _{m_{k}}\cap \Upsilon _{j}\neq \emptyset .
$$

Let $S_{1}:=\{i\in \{1,\dots,n\}\ \vert \ \Upsilon _{i}\sim \Upsilon _{1}\}$; define
$\Lambda _{1}:=\cup _{i\in S_{1}}\Upsilon _{i}.$ Set $S_2=\{1,\dots, n\}\setminus S_1$
and $
\widetilde{\Lambda }_{2}:=\cup _{i\in S_{2}}\Upsilon _{i}.$
Then $\Lambda =\Lambda _{1}\cup \widetilde{\Lambda }_{2}$ with $\Lambda _{1}$
non-fragmentable. If $\widetilde{\Lambda }_{2}=\cup _{i\in S_{2}}\Upsilon
_{i}$ is non-fragmentable then by defining $\widetilde{\Lambda }_{2}=\Lambda
_{2}$ we have that $\Lambda =\Lambda _{1}\cup \Lambda _{2}$ is the optimal
fragmentation of $\Lambda =\cup_{i=1}^n\Upsilon _{i}.$
Otherwise, $\widetilde{\Lambda }_{2}=\cup _{i\in S_{2}}\Upsilon _{i}$ is
fragmentable and, as before, we may decompose $\widetilde{\Lambda }_{2}:=\Lambda _{2}\cup
\widetilde{\Lambda }_{3},$ with $\Lambda _{2}$  non-fragmentable. By
reiterating the process we obtain a decomposition $\Lambda =\cup_{i=1}^k\Lambda _i$, where every $\Lambda_i$ is non-fragmentable.
This produces an optimal fragmentation $\Lambda =\cup_{i=1}^k\Lambda_i$
 of the initial decomposition $\Lambda =\cup_{i=1}^n\Upsilon_i.$
\medskip

\begin{proposition}
\label{uni-fragme}Let $\Lambda $ be a finite set and let $\Upsilon
_{1},\dots,\Upsilon _{n}$ be non-empty subsets of $\Lambda $ $\ $such that $%
\Lambda =\bigcup\limits_{i=1}^{n}\Upsilon _{i}$ is a fragmentable union.
Then the optimal fragmentation $\Lambda =\Lambda _{1}\cup \dots\cup \Lambda
_{k}$ of $\Lambda =\bigcup\limits_{i=1}^{n}\Upsilon _{i}$ is unique (unless
reordering).
\end{proposition}

\begin{proof}
Suppose that $\Lambda =\Lambda _{1}\cup \dots \cup \Lambda _{k}$ and $\Lambda =%
\widetilde{\Lambda }_{1}\cup \dots \cup \widetilde{\Lambda }_{m}$ are two
optimal fragmentations. Take $i\in \{1,\dots,k\}.$ If there exist $j,k\in
\{1,\dots,m\}$ such that $\Lambda _i \cap \widetilde{\Lambda }_{j}\neq
\emptyset $ and $\Lambda _{i}\cap \widetilde{\Lambda }_{k}\neq \emptyset ,$
then $j=k$ because, otherwise, $\Lambda _{i}$ is fragmentable. It
follows that for every $i\in \{1,\dots,k\}$ there is a unique $j\in \{1,\dots,m\}$
such that $\Lambda _{i}\subseteq \widetilde{\Lambda }_{j}.$ We claim that $\Lambda_i = \widetilde\Lambda_j$ because otherwise $\widetilde{
\Lambda }_{j}$ would be fragmentable, a contradiction. \ We conclude that $m=k$ and that $\Lambda =%
\widetilde{\Lambda }_{1}\cup \dots \cup \widetilde{\Lambda }_{m}$ is a
reordering of $\Lambda =\Lambda _{1}\cup \dots \cup \Lambda _{k}.$
\end{proof}

By combining Theorems \ref{CharacSimple} and \ref{caracteriz}
with the optimal fragmentation process we obtain the following result.

\medskip

\begin{theorem}
\label{final}Let $A$ be a finite-dimensional evolution algebra with natural basis $B=\{e_{i}\ \vert \ i \in \Lambda \}$. Let $\{C_{1},\dots,C_{k}\}$ be the set of principal
cycles of $\Lambda ,$ $\{i_{1},\dots,i_{m}\}$ the set of all chain-start
indices of $\Lambda$ and consider the canonical decomposition
\begin{equation*}
(\dag) \quad \Lambda =\Lambda (C_{1})\cup \dots\cup \Lambda (C_{k})\cup \Lambda (i_{1})\cup \dots \cup
\Lambda (i_{m}).
\end{equation*}
Let $\Lambda= \sqcup_{\gamma\in \Gamma}\Lambda_\gamma$ be the optimal fragmentation of $(\dag)$ and decompose $B=  \sqcup_{\gamma\in \Gamma}B_\gamma$, where $B_\gamma = \{e_{i}\ \vert \ i \in \Lambda_\gamma\}$. Then $A= \oplus_{\gamma\in \Gamma}I_\gamma$, for $I_\gamma= {\rm lin}\ B_\gamma$, which is an evolution ideal of $A$. Moreover, if $A$ is non-degenerate, then $A= \oplus_{\gamma\in \Gamma}I_\gamma$ is the optimal direct-sum decomposition of $A$.
\end{theorem}

Since the optimal direct-sum decomposition of a non-degenerate
evolution algebra $A$ is unique, we conclude that, in the non-degenerate
case, the decomposition given in Theorem \ref{final} does not
depend on the prefixed natural basis $B$ (i.e. any other natural basis leads
to the same optimal direct sum decomposition).

\begin{remark}
\rm
Every  finite dimensional evolution algebra $A$ (non-degenerated or not) is the direct sum of a finite number
of irreducible evolution algebras.
Indeed, if $A$ is irreducible, then we are done. Otherwise, decompose $A=I_1\oplus I_2$, for $I_1, I_2$  ideals of $A$. If $I_1$ and $I_2$ are irreducible, then we have finished. If this is not the case, we decompose them. Since the dimension of $A$ is finite, proceeding in this way in a finite number of steps we finish.
\end{remark}

For $A$ an evolution algebra of arbitrary dimension such that  $A=\oplus_{\gamma \in \Gamma}I_\gamma$ is the optimal direct-sum decomposition of $A$, the
study of $A$ can be reduced to the study of the irreducible evolution
algebras $I_\gamma$ separately.
\medskip

The last result in this section is a consequence of Proposition \ref{simple} and Theorem \ref{final}.

\begin{corollary}\label{suma-simple}
Let $A$ be a non-degenerate finite dimensional evolution algebra with a natural basis $B=\{e_{i}\ \vert \ i \in \Lambda \}$.
Then $A=I_{1}\oplus\dots\oplus I_{k}$, where $I_i$ is an ideal, simple as an algebra,  if and only if $\Lambda $ has the following property:  $\Lambda =\Lambda _{1}\sqcup \dots \sqcup\Lambda _{k}$, where every $\Lambda_j $ is non-empty  and $I_{j}={\rm lin}\{e_i \ \vert\ i\in \Lambda_j\}={\rm lin}\{e_i^2 \ \vert \ i\in \Lambda_j\}$ and $D(i)=\Lambda_{j},$ for every $i\in\Lambda_j$.
\end{corollary}

\section{The optimal fragmentation computed with Mathematica}

We have designed a program that provides the optimal fragmentation of an evolution algebra when we introduce the coefficient matrix as input. Moreover, the code identifies if an index is a cyclic-index, a principal cyclic index or a chain-start index. On the other hand, it calculates the nth-generation descendents of any index for every n.
We include the \emph{Mathematica} codes needed for our computations. They
consist on a list of functions written in the order they have been used. The computation of the invariants has
been performed by the \emph{Mathematica} software.

In order to compute the optimal fragmentation, we have used the proposition that follows.


\begin{proposition}
Let $\Lambda $ be a finite set and let $\Upsilon _{1},...,\Upsilon _{n}$ be
non-empty subsets of $\Lambda $ such that $\Lambda = \bigcup\limits_{i=1}^{n}\Upsilon _{i}$. Let  $(a_{ij})\in M_{n}({\mathbb{K}})$ be
the matrix defined by:  $a_{ii}=0$ for every $i$, $a_{ij}=1$ if $\Upsilon _{i}\cap \Upsilon _{j}\neq
\emptyset $ and $a_{ij}=0$ if $\Upsilon _{i}\cap \Upsilon
_{j}=\emptyset $. Let $E$ be the graph whose adjacency matrix is
$(a_{ij})$. Then, $E$ is connected if and only if $\Lambda
=\bigcup\limits_{i=1}^{n}\Upsilon _{i}$ is not a fragmentable union.
Moreover, if $E$ is not connected, then the connected components of $E$
form an optimal fragmentation of $\Lambda $.
\end{proposition}

\begin{proof}

Suppose that $\Lambda =\bigcup\limits_{i=1}^{n}\Upsilon _{i}$ is a non
fragmentable union. If $E$ is not connected, let $\Psi _{i}$ denote the connected components of $E$ with $i\in \{1,2,\ldots ,m\}$ for some $m\in \N$. This
means that we may write $\{1,2,\ldots, n\}=\bigsqcup\limits_{i=1}^{m}\Psi_{i}$ where $\Psi _{i}\subseteq \{1,2,\ldots
,n\}$. Now, we consider the sets: $\Lambda _{i}=\bigcup\limits_{j\in \Psi
_{i}}\Upsilon _{j}$. We will show that $\Lambda
=\bigcup\limits_{i=1}^{m}\Lambda _{i}$ is an optimal fragmentation. First,
we have to prove that $\Lambda _{i}\cap \Lambda _{j}=\emptyset $ for every $i\neq j$, with $i,j\in \{1,\ldots ,m\}$. If there exists $\omega \in \Lambda _{i}\cap\Lambda _{j}$, then there are $r\in \Psi _{i}$ and $s\in \Psi _{j}$ such that $\omega \in \Upsilon _{r}\cap \Upsilon _{s}$. This implies that $\Upsilon_{r}\cap \Upsilon _{s}\neq \emptyset $, i.e. $r$ and $s$ are connected. This is a contradiction because they belong to different connected components.
Conversely, suppose that $E$ is connected. If $\Lambda=\bigcup\limits_{i=1}^{n}\Upsilon _{i}$ is a fragmentable union then there exist $\Lambda _{1}$ and $\Lambda _{2}$ disjoint subsets of $\Lambda $ satisfying that $\Lambda =\Lambda _{1}\cup \Lambda _{2}$ and such that for every $i=1,\ldots ,n$, either $\Upsilon _{i}\subseteq \Lambda _{1}$ or $\Upsilon _{i}\subseteq \Lambda _{2}$. Let $\alpha \in \Lambda _{1}$ and $\beta \in \Lambda _{2}$. This means that there exist $i,j\in \{1,\ldots ,n\}$ such that $\alpha \in \Upsilon _{i}\subseteq \Lambda _{1}$ and $\beta \in \Upsilon _{j}\subseteq \Lambda _{2}$. As $E$ is connected, there exists a path from $\alpha $ to $\beta $. This implies that there exist $i_{1},\ldots ,i_{k}\in \{1,\ldots ,n\}$ such that
$$
\Upsilon _{i}\cap \Upsilon _{i_{1}}\neq \emptyset ,\Upsilon _{i_{1}}\cap
\Upsilon _{i_{2}}\neq \emptyset ,\ldots, \Upsilon _{i_{k}}\cap \Upsilon
_{j}\neq \emptyset,
$$
a contradiction because $\alpha \in \Upsilon _{i}\subseteq \Lambda _{1}$ and $\beta \in
\Upsilon _{j}\subseteq \Lambda _{2}$. Furthermore, from this reasoning we deduce that the connected components of $E$ make an optimal fragmentation of $\Lambda $.
\end{proof}


In what follows we provide  a list with the routines that have been used together with a brief description of them.

\begin{itemize}
\item {\bf D}$_1$: computes the first-generation descendents of i.
\item {\bf D}$_n$: computes the nth-generation descendents of i.
\item {\bf CycleQ}: checks if P has a cycle.
\item {\bf DP}: computes D(i).
\item {\bf CyclicQ}: checks if P has some cyclic index.
\item {\bf CycleAssociated}: computes the cycle associated to i.
\item {\bf Ascendents}: computes the ascendents of i.
\item {\bf PrincipalCycleQ}: checks if i is a principal cyclic-index.
\item {\bf ChainStartQ}: checks if i is a chain-start index.
\item {\bf CanonicalDecomposition}: computes a canonical decomposition associated to P.
\item {\bf OptimalFragmentation}: computes an optimal fragmentation associated to P.
\end{itemize}

Finally, we include the \emph{Mathematica} code of all these functions.

\medskip

\def\peq{\fontsize{9}{9}\selectfont}

\medskip

{\peq
\begin{tabular}{l}
$l={\bf Table}[i,\{i,n\}];$\\
\hskip .5cm ${\rm {\bf D}_1 [i\_,P\_]:= {\bf Select}[l,P[[\#,i]]\neq 0 \&];}$\\
\end{tabular}
}

\medskip

{\peq
\begin{tabular}{l}
${\rm {\bf D}_{n\_} [i\_, P\_ ]:= {\bf Module}[\{j, a, s\},}$\\
\hskip .5cm ${\rm a = \{\}; s={\bf Length}[{\bf D}_{n-1}[i,P]];}$ \\
\hskip .5cm ${\rm {\bf If}[n==1, {\bf D}_1[i,P],}$ \\
\hskip .5cm ${\rm {\bf Union}[{\bf Flatten}[{\bf Table}[{\bf D}_1[{\bf D}_{n-1}[i,P][[t]],P],\{t,{\bf Length}[{\bf D}_{n-1}[i,P]]\}]]]]]}$ \\
\end{tabular}
}

\medskip

{\peq
\begin{tabular}{l}
${\rm {\bf CycleQ}[P\_ ]:= {\bf Module}[\{n, a\},n={\bf Length}[P];}$\\
\hskip .5cm ${\rm a = {\bf Union}[{\bf Flatten}[{\bf Table}[}$ \\
\hskip .5cm ${\rm {\bf Diagonal}[{\bf MatrixPower}[P,i]],\{i,1,n\}]]];}$ \\
\hskip .5cm ${\rm {\bf MemberQ}[a,1];}$ \\
\end{tabular}
}

\medskip

{\peq
\begin{tabular}{l}
${\rm {\bf DYesCycle}[i\_,P\_ ]:= {\bf Module}[\{j, a\},}$\\
\hskip .5cm ${\rm a = \{\};}$ \\
\hskip .5cm ${\rm {\bf For}[j=1,j<={\bf Length}[P],j++,}$ \\
\hskip .5cm ${\rm {\bf AppendTo} [a,{\bf D}_{j}[i,P]]];}$ \\
\hskip .5cm ${\rm {\bf Apply}[{\bf Union},a]]}$\\
\end{tabular}
}

\medskip

{\peq
\begin{tabular}{l}
${\rm {\bf DNotCycle}[i\_,P\_ ]:= {\bf Module}[\{j, a\},}$\\
\hskip .5cm ${\rm a = \{{\bf D}_1[i,P]\};}$ \\
\hskip .5cm ${\rm {\bf For}[j=1,{\bf D}_{j}[i,P]\neq {\bf D}_{j+1}[i,P],j++,}$ \\
\hskip .5cm ${\rm {\bf AppendTo} [a,{\bf D}_{j+1}[i,P]]];}$ \\
\hskip .5cm ${\rm {\bf Apply}[{\bf Union},a]]}$\\
\end{tabular}
}

\medskip

{\peq
\begin{tabular}{l}
${\rm {\bf DP} [i\_, P\_ ]:= {\bf If}[{\bf CycleQ}[P],{\bf DYesCycle}[i,P],{\bf DNotCycle}[i,P]]}$ \\
\end{tabular}
}

\medskip

{\peq
\begin{tabular}{l}
${\rm {\bf CyclicQ} [i\_, P\_ ]:= {\bf If}[{\bf MemberQ}[{\bf DP}[i,P],i],}$ \\
\hskip .5cm {\rm {\bf Print}[i ``is a cyclic index"],{\bf Print}[i ``is not a cyclic index"]]}\\
\end{tabular}
}

\medskip
{\peq
\begin{tabular}{l}
${\rm {\bf CycleAssociated}[i\_,P\_ ]:= {\bf Module}[\{j, a\},}$\\
\hskip .5cm ${\rm a = \{\};}$ \\
\hskip .5cm ${\rm {\bf For}[j=1,j<={\bf Length}[P],j++,}$ \\
\hskip .5cm ${\rm {\bf If}{\bf MemberQ}[{\bf DP}[i,P],j] \& \& {\bf MemberQ}[{\bf DP}[j,P],i],}$ \\
\hskip .5cm ${\rm {\bf AppendTo} [a,j]]];}$ \\
\hskip .5cm ${\rm a]}$\\
\end{tabular}
}

\medskip

{\peq
\begin{tabular}{l}
${\rm {\bf Ascendents}[i\_,P\_ ]:= {\bf Module}[\{j, a\},}$\\
\hskip .5cm ${\rm a = \{\};}$ \\
\hskip .5cm ${\rm {\bf For}[j=1,j<={\bf Length}[P],j++,}$ \\
\hskip .5cm ${\rm {\bf If}{\bf MemberQ}[{\bf DP}[j,P],i],}$ \\
\hskip .5cm ${\rm {\bf AppendTo} [a,j]]];}$ \\
\hskip .5cm ${\rm a]}$\\
\end{tabular}
}

\medskip

{\peq
\begin{tabular}{l}
${\rm {\bf Subset}[A\_,B\_ ]:= ({\bf Union}[A,B]=={\bf Union}[B])}$
\end{tabular}
}

\medskip

{\peq
\begin{tabular}{l}
${\rm {\bf PrincipalCycleQ} [i\_, P\_ ]:= {\bf If}[{\bf Subset}[{\bf Ascendents }[i,P],{\bf CycleAssociated}[i,P]],}$ \\
\hskip .5cm {\rm {\bf Print}[i ``is a principal cyclic-index"],} \\
\hskip .5cm {\rm {\bf Print}[i ``is not a principal cyclic-index"]]}
\end{tabular}
}

\medskip

{\peq
\begin{tabular}{l}
${\rm {\bf ElementsNotNoneRow}[P\_]:= {\bf Module}[\{j\},}$\\
\hskip .5cm ${\rm {\bf Select}[{\bf Table}[j,\{j,{\bf Length}[P]\}],P[[\#]]== 0 P[[1]]\&]];}$ \\
\end{tabular}
}

\medskip

{\peq
\begin{tabular}{l}
${\rm {\bf ChainStartQ} [i\_, P\_ ]:= {\bf If}[{\bf MemberQ}[{\bf ElementsNotNoneRow}[P],i],}$\\
\hskip .5cm {\rm {\bf Print}[i `` is a  chain-start index"],} \\
\hskip .5cm {\rm {\bf Print}[i ``is not a chain-start index"]]}\\
\end{tabular}
}

\medskip

{\peq
\begin{tabular}{l}
${\rm {\bf \Lambda} [i\_, P\_ ]:= {\bf Union}[\{i\},{\bf DP}[i,P]];}$\\
\end{tabular}
}

\medskip

{\peq
\begin{tabular}{l}
${\rm {\bf LambdaChainStart} [P\_ ]:={\bf Table}[{\bf \Lambda}[{\bf ElementsNotNoneRow}[P][[i]],P], \{i,{\bf Length}[{\bf ElementsNotNoneRow}[P]]\}]}$
\end{tabular}
}

\medskip

{\peq
\begin{tabular}{l}
${\rm {\bf LambdaPrincipalCycle}[P\_ ]:= {\bf Module}[\{j, a\},}$\\
\hskip .5cm ${\rm a = \{\};}$ \\
\hskip .5cm ${\rm {\bf For}[j=1,j<={\bf Length}[P],j++,}$ \\
\hskip .5cm ${\rm {\bf If}[{\bf Subset}[{\bf Ascendents }[j,P],{\bf CycleAssociated}[j,P]],}$ \\
\hskip .5cm ${\rm {\bf AppendTo} [a,{\bf DP}[j,P]]]];}$ \\
\hskip .5cm ${\rm a]}$\\
\end{tabular}
}

\medskip

{\peq
\begin{tabular}{l}
${\rm {\bf CanonicalDecomposition} [P\_ ]:= {\bf Join}[{\bf LambdaChainStart}[P], {\bf LambdaPrincipalCycle}[P]]}$\\
\end{tabular}
}

\medskip

{\peq
\begin{tabular}{l}
${\rm {\bf f} [i\_, j\_ , P\_ ]:= {\bf If}[[i==j,0,}$\\
\hskip .5cm ${\rm {\bf If}[{\bf Intersection}[{\bf Part}[{\bf CanonicalDecomposition [P]},i],}$ \\
\hskip .5cm ${\rm {\bf Part}[{\bf CanonicalDecomposition [P]},j]]\neq 0,1,0]]}$
\end{tabular}
}

\medskip

{\peq
\begin{tabular}{l}
${\rm {\bf Matr}[P\_ ]:= {\bf Table}[}$\\
\hskip .5cm ${\rm {\bf f}[i,j,P],\{i,{\bf Length}[{\bf CanonicalDecomposition[P]}]\},\{j}$ \\
\hskip .5cm ${\rm {\bf Length}[{\bf CanonicalDecomposition[P]}]\}]}$
\end{tabular}
}

\medskip

{\peq
\begin{tabular}{l}
${\rm {\bf OptimalFragmentation }[P\_ ]:= {\bf ConnectedComponents}[{\bf AdjacencyGraph}[{\bf Matr[P]}],{\bf VertexLabels}\rightarrow {\bf "Name"}]}$\\
\end{tabular}
}

One concrete example showing how this program works can be found in
\url{https://www.dropbox.com/s/2mtdojjaj1o20m8/OptimalFragmentation.pdf?dl=0}. We have not included it here to not enlarge the paper.

\medskip


\section*{Acknowledgments}
The authors would like to thank all the participants in the M\'alaga permanent seminar ``Estructuras no asociativas y \'algebras de caminos de Leavitt" for useful discussions during the preparation of this paper, in particular to Prof. C\'andido Mart\'\i n Gonz\'alez for his helpful comments. They also would like to thank the referee for useful suggestions which have enriched the paper.
\medskip

The three authors have been supported by the Junta de Andaluc\'{\i}a and Fondos FEDER, jointly, through projects  FQM 199 (the third author) and FQM-336 and FQM-7156 (the first and second authors). The two first authors have been also partially supported by the Spanish Ministerio de Econom\'ia y Competitividad and Fondos FEDER, jointly, through project  MTM2013-41208-P.

\end{document}